\documentclass[12pt,twoside]{article}
\usepackage[]{amssymb, latexsym, amsmath, diagrams, hyperref}
\diagramstyle[labelstyle=\scriptstyle]

\newtheorem{theorem}{Theorem}[section]
\newtheorem{lemma}[theorem]{Lemma}

\newenvironment{proof}[1][Proof]{\begin{trivlist}
\item[\hskip \labelsep {\bfseries #1}]}{\end{trivlist}}
\newenvironment{definition}[1][Definition]{\begin{trivlist}
\item[\hskip \labelsep {\bfseries #1}]}{\end{trivlist}}
\newenvironment{example}[1][Example]{\begin{trivlist}
\item[\hskip \labelsep {\bfseries #1}]}{\end{trivlist}}
\newenvironment{remark}[1][Remark]{\begin{trivlist}
\item[\hskip \labelsep {\bfseries #1}]}{\end{trivlist}}

\newcommand{\qed}{\nobreak \ifvmode \relax \else
      \ifdim\lastskip<1.5em \hskip-\lastskip
      \hskip1.5em plus0em minus0.5em \fi \nobreak
      \vrule height0.75em width0.5em depth0.25em\fi}

\DeclareSymbolFont{AMSb}{U}{msb}{m}{n}
\DeclareMathSymbol{\N}{\mathbin}{AMSb}{"4E}
\DeclareMathSymbol{\Z}{\mathbin}{AMSb}{"5A}
\DeclareMathSymbol{\R}{\mathbin}{AMSb}{"52}
\DeclareMathSymbol{\Q}{\mathbin}{AMSb}{"51}
\DeclareMathSymbol{\I}{\mathbin}{AMSb}{"49}
\DeclareMathSymbol{\C}{\mathbin}{AMSb}{"43}

\newcommand{\pro}{\mathbb{P}}

\begin{document}

\title{Genus-Zero Mirror Principle For Two Marked Points}
\author{Luke Cherveny}
\date{December 27, 2009}

\maketitle

\begin{abstract}

We study a generalization of Lian-Liu-Yau's notion of Euler data \cite{lly:mirror1} in genus zero and show that certain sequences of multiplicative equivariant characteristic classes on Kontsevich's stable map moduli with markings induce data satisfying the generalization.  In the case of one or two markings, this data is explicitly identified in terms of hypergeometric type classes, constituting a complete extension of Lian-Liu-Yau's mirror principle in genus zero to the case of two marked points and establishing a program for the general case.  We give several applications involving the Euler class of obstruction bundles induced by a concavex bundle on $\pro^n$.  

\end{abstract}

\section{Introduction}

Given $X$ be a smooth projective variety over $\C$, let $\overline M_{g,m}(X,\beta)$ be Kontsevich's moduli space of stable maps to $X$.  Its points are triples $(f;C;z_1,...,z_m)$, representing a holomorphic map $f:C \rightarrow X$ from a genus $g$ curve with smooth markings $z_1,...,z_m \in C$ and such that $f_*[C] = \beta \in H_2(X,\Z)$, modulo the obvious equivalence.  For an excellent introduction to stable maps see \cite{fultonpand:stablemaps}.  The Gromov-Witten theory of $X$ with primary insertions concerns integrals of the form

\begin{equation}
K_{g,d}^X(\gamma_1, ..., \gamma_m) = \int_{[\overline M_{g,m}(X,\beta)]^{\text{vir}}} \prod_{i=1}^m ev_i^*\gamma_i,
\label{gwdef}
\end{equation}

\noindent where integration is over the virtual fundamental class $[\overline M_{g,m}(X,\beta)]^{\text{vir}}$ of Li-Tian \cite{litian:virtualclass} (also Behrend-Fantechi \cite{behrendfantechi:virtualclass}), $ev_i$ is the map 

\[
ev_i: \overline M_{g,m}(X,\beta) \rightarrow X
\]

\noindent given by evaluation of $f$ at the $i$-th marked point, and $\gamma_i \in H^*(X,\Z)$.  As $\overline{M_{g,m}}(X,\beta)$ is a Deligne-Mumford stack, the Gromov-Witten invariants are in general $\Q$-valued.  

When $X$ is Calabi-Yau, the duality between type IIA and IIB string theories postulates that a certain generating series for the Gromov-Witten invariants of $X$ (the A-model partition function) is computable in terms of period integrals on the complex moduli of a mirror space (the B-model), for instance yielding the mirror conjectures of \cite{cdgp} and \cite{bcov} for the quintic.  For a thorough history and overview of mirror symmetry, see either \cite{coxkatz} \cite{mirrorsymmetry}.  We will restrict our attention to the mirror principle program of Bong Lian, Kefeng Liu, and S.-T. Yau which seeks to prove mirror conjectures in a simple and unified framework by exhibiting certain series of equivariant multiplicative characteristic classes on Kontsevich's moduli of stable maps in terms of hypergeometric type classes \cite{lly:mirror1} \cite{lly:mirror2} \cite{lly:mirror3} \cite{lly:mirror4}.

The original paper \cite{lly:mirror1} develops mirror principle in genus zero for $X$ a complete intersection and more generally a local Calabi-Yau space arising from a concavex bundle $V$ over $\pro^n$ (direct sum of a positive and negative bundles).  Such a bundle $V = V^+ \oplus V^-$ naturally induces a sequence of obstruction bundles $V_d$ on $\overline{M_{0,m}}(\pro^n,d)$ whose fiber at $(f;C;z_1,...z_m)$ is given by $H^0(C,f^*V^+) \oplus H^1(C,f^*V^-)$.  The calculation of Gromov-Witten invariants (\ref{gwdef}) for $X$ reduces to enumeration of twisted Euler classes of these obstruction bundles:

\begin{equation}
K_{d}^X(H^{k_1},...,H^{k_m}) = \int_{\overline{M_{0,m}}(\pro^n,d)} \text{Euler}(V_d) \prod_{i=1}^m ev_i^*H^{k_i}
\label{gwdefinition2}
\end{equation}

\noindent Here $H$ is the hyperplane class on $\pro^n$ and we have identified $\beta = d[H]$.  The power of this formulation is that a torus $T$ acting on $\pro^n$ induces a $T$-action on $\overline{M_{0,m}}(\pro^n,d)$, which permits the localization methods of Atiyah-Bott \cite{atiyahbott:mm} to be applied to an equivariant version of (\ref{gwdefinition2}).  We will review the essentials of equivariant cohomology and localization in Section 2, as well as introduce Lian-Liu-Yau's notion of Euler data and give a more precise formulation of their program.  The idea is that one should study sequences of equivariant classes on a large projective space (the so-called linear model) that satisfy certain quadratic relations on fixed points under a torus action.  These relations reflect the geometry of fixed components of a complicated moduli space (the nonlinear model) which collapses onto the linear model.  One instance of Euler data is induced by the concavex bundle $V$ and corresponds to the $A$-model; another instance, corresponding to the $B$-model, is given in terms of linear classes on the linear model.  Their equality after certain mirror transformations is established in \cite{lly:mirror1} in the case $m=0$.  When $m>0$, Lian-Liu-Yau give a general formulation of Euler data but do not explicitly identify the proper linear classes for comparison \cite{lly:mirror4}.

In section $3$ we introduce the notion of $(u,v)$-Euler data, which can be interpreted as a deformation of Lian-Liu-Yau's Euler data (seen to be the $(0,0)$ case).  We show that a concavex bundle on $\pro^n$ indeed induces $(u,v)$-Euler data, and prove a key uniqueness lemma for the specific case of $(0,1)$-Euler data.  Section $4$ gives a natural extension of $(0,0)$-Euler data to $(0,1)$-Euler data and discusses certain invertible maps on $(u,v)$-Euler data known as mirror transformations.  The goal is to prove Theorem \ref{maintheorem}, which gives an extension to $(0,1)$-Euler data with controlled growth in a certain equivariant weight critical to our uniqueness lemma.  In Section $5$ we combine the results of previous sections to prove Theorems \ref{oneptmirrorthrm} and \ref{twoptmirrorthrm}, the one-point and two-point mirror theorems establishing the equality of Euler data induced by $V$ with explicite linear data.  When $V = \mathcal{O}(n+1)$ on $\pro^n$ and working with the equivariant Euler class this agrees with the results of \cite{zinger:genuszero}.  Section $6$ discusses the explicit recovery of Gromov-Witten invariants in general, while Section $7$ gives several concrete examples.

We note that integrality conjectures have been advanced connecting the numbers produced in (\ref{gwdef}) with the enumerative geometry of $X$.  In dimension $3$, the Aspinwall-Morrison formula \cite{aspinwallmorrison:covers} conjectures integer invariants in the case $g=0$, $m=0$ when $X$ is a Calabi-Yau after multiple coverings are properly taken into account.  This was generalized by Gopakumar and Vafa to a full integratlity conjecture in all genera for CY $3$-folds.  New integrality conjectures for Calabi-Yau spaces in genus $1$ for dimensions $4$ and $5$ and an extension of the Aspinwall-Morrison formula to $n$ dimensions have also recently been advanced \cite{klemmpand:4folds} \cite{pandzinger:5folds}.  To be specific, in genus $0$ it is conjectured for Calabi-Yau $n$-folds that the invariants $\eta$ defined via

\begin{equation}
\sum_{\beta > 0} K_d(H^{k_1},...,H^{k_m}) q^d = \sum_{\beta > 0} \eta_d(H^{k_1},...,H^{k_m}) \sum_{d = 1}^{\infty} d^{-3+m} q^{d\beta}
\label{aspinwallmorrison}
\end{equation}

\noindent are integers.  The author has verified the integrality of these invariants in $m=1$ or $m=2$ for a number of different $V$ in low degree by computer.\\

\noindent \textbf{Acknowledgments} The author wishes to thank his advisor, Prof. Kefeng Liu, for his advice and support, as well as Prof. Jun Li for his interest in the project and assistance.

\section{Preliminaries}

\subsection{Equivariant Cohomology and Localization}

We briefly summarize equivariant cohomology for $T$-spaces and the powerful technique of localization.  For a thorough discussion, the standard reference is \cite{atiyahbott:mm}.

Let T be a compact Lie group.  There exists a contractible space $ET$, unique up to homotopy equivalence, on which $T$ acts freely, and with $T$ classified by the principle bundle $ET \rightarrow BT$.  If $X$ is a topological space with $T$-action, we define $X_T = X \times_T ET$, which is a bundle over the classifying space $BT$ with fiber $X$.

\begin{definition} The \emph{equivariant cohomology} of $X$ is defined to be

\[
H^*_T(X) = H^*(X_T),
\]

\noindent where $H^*(X_T)$ is the ordinary cohomology of $X_T$.
\end{definition}

We will always take $T$ to be an algebraic torus $(\C^*)^n$, in which case one may check that $BT \cong (\pro^{\infty})^n$ and $ET = \pi_1^*S \otimes \dots \otimes \pi_n^*S$, where $\pi_i$ is projection of $BT$ to the $i$-th copy of $\pro^{\infty}$ while $S \cong \mathcal{O}(-1)$ is the tautological line bundle on $\pro^{\infty}$.  When $X$ is a point we recover the ordinary cohomology of the classifying space $BT$:

\[
H^*_T(pt) = H^*(BT) = \C[\lambda_1, \dots, \lambda_n],
\]

\noindent where $\lambda_i = c_1(\mathcal{O}(\lambda_i))$ are the weights of the torus action.

Many of the standard concepts in ordinary cohomology translate directly to the equivariant setting, for instance the notions of pullback for $T$-maps and pushforward by a proper $T$-map.  If we denote inclusion of the fiber by  $i_X: X \hookrightarrow X_T$, then in particular the equivariant pullback induces a map $i_X^*: H_T^*(X) \rightarrow H^*(X)$ called the \emph{nonequivariant limit}.  Furthermore, pullback by the contraction map $X \rightarrow pt$ realizes $H^*_T(X)$ as a $H^*(BT)$-module.  The notions of equivariant vector bundle and equivariant characteristic classes are also readily defined.

\subsection{Atiyah-Bott Localization}

If $X$ be a smooth manifold acted on by a torus $T$, the $T$-fixed components of $X$ are known to be a union of smooth submanifolds $\{Z_j\}$ (e.g. \cite{iversen:fixedpoints}).  We will denote inclusion of a fixed component into $X$ by $i_{Z_j}: Z_j \hookrightarrow X$ and let the (equivariant) normal bundle of $Z_j \subseteq X$ be denoted by $N_j$.  A key observation is that the equivariant Euler class $\text{Euler}_T(N_j) \in H_T^*(Z_j)$, which we may pushforward by the Gysin map $i_{Z_j!}: H^*_T(Z_j) \rightarrow H^*_T(X)$, is invertible in $H^*_T(X)\otimes \mathcal{R}_T$ where $\mathcal{R}_T \cong \C(\lambda_1,...,\lambda_n)$ is the field of fractions of $H^*(BT)$.

A fundamental result in the subject is that there is an isomorphism between the equivariant cohomology of $X$ properly localized and that of its fixed components

\[
H_T^*(X)\otimes \mathcal{R}_T \cong \bigoplus_j H_T^*(Z_j)\otimes \mathcal{R}_T,
\]

\noindent given explicitly by the map

\[
\alpha \rightarrow \sum_j \frac{i_{Z_j}^*\alpha}{\text{Euler}(N_{Z_j})}.
\]

\noindent In integrated form, 

\[
\int_{X_T} \alpha = \sum_{Z_j} \int_{(Z_j)_T} \frac{i_{Z_j}^*\alpha}{\text{Euler}(N_{Z_j})}
\]

\noindent for any $\alpha \in H_T^*(X)\otimes\mathcal{R}_T$.

More generally, let $f: X \rightarrow Y$ be a proper equivariant map of $T$-spaces.  If $F \subseteq Y$ is a $T$-fixed component of $Y$, $E \subset f^{-1}(F)$ fixed components of $f^{-1}(F)$, and $f_0 := f|_{E}$, then \emph{functorial localization} \cite{lly:mirror1} says for $\omega \in H_T^*(X)$,

\[
{f_0}_* \left(\frac{i_E^*\omega}{e_T(E/X)} \right) = \frac{i^*_F(f_*\omega)}{e_T(F/Y)}.
\]

Localization techniques hold in the case where the spaces support support virtual fundamental classes.  For instance, see the discussion in Ch. 27 of \cite{mirrorsymmetry}.

\subsection{Lian-Liu-Yau Euler Data}

We now give a brief overview of Lian, Liu, and Yau's mirror principle, in particular focusing on their notion of Euler data.  Ultimately one wishes to establish the equivalence of two sequences of equivariant classes on projective spaces satisfying certain gluing relations on the fixed points.  Such sequences are known as Euler data and are the primary item of interest in mirror principle.  We refer the reader to the original paper for further discussion \cite{lly:mirror1}.

Let $M_d = \overline{M_{0,0}}(\pro^n \times \pro^1, (d,1))$. This space compactifies the space of all unmarked rational maps of degree $d$ to $\pro^n$ and is often called the \emph{nonlinear model} (the map to $\pro^n$ is the image of the the map given by each point in $M_d$).  $M_d$ has a $G = \C^* \times T$-action, where $T=(\C^*)^{n+1}$, induced by the $G$-action 

\begin{align}
(t, t_0, \dots, t_n) \cdot& ([w_0,w_1], [z_0,...,z_n]) = ([t^{-1}w_0,w_1],[t_0^{-1}z_0, \dots, t_n^{-1}z_n])
\label{gaction}
\end{align}

\noindent on $\pro^1 \times \pro^n$.  Furthermore, there is a natural (equivariant) forgetful morphism $\pi$ to $\overline{M_{0,0}}(\pro^n, d)$:

\[
\pi: M_d \rightarrow \overline{M_{0,0}}(\pro^n, d).
\]

Also define $N_d = \pro(H^0(\pro^1, \mathcal{O}(d))^{n+1})$.  Since an unmarked degree $d$ rational map to $\pro^n$ may be given up to scalar multiple as an $(n+1)$-tuple of degree $d$ homogeneous polynomials in two variables with no common factor, $N_d$ constitutes a different compactification of unmarked rational curves to $\pro^n$ of degree $d$, which we call the \emph{linear model}.  It too has a $G$-action, given by

\[
(t, t_0, \dots, t_n) \cdot [\alpha_0(w_0,w_1), \dots, \alpha_n(w_0,w_1)] = [t_0^{-1}\alpha_0(t^{-1}w_0,w_1), \dots, t_n^{-1}\alpha_n(t^{-1}w_0,w_1)]
\]

\noindent where the $\alpha_i$ are degree $d$ homogeneous polynomials in two variables.

The $G$-fixed components of $N_d$ are isolated fixed points, which we denote by $\{p_{i,r}\}$ where $0 \leq i \leq n$ and $0 \leq r \leq d$, representing the point $[0,...,w_1^rw_0^{d-r},...,0]$ where the nonzero term is in the $i$-th slot.  We similarly label the $T$-fixed points of $\pro^n$ by $\{ q_i \}$ where $0 \leq i \leq n$.  One may compute the equivariant cohomology ring of $N_d$ to be

\[
H^*_G(N_d) = \frac{\C[\kappa, \lambda_0, ..., \lambda_n, \hbar]}{\prod_{l=0}^n \prod_{m=1}^d (\kappa - \lambda_l - m\hbar)}
\]

\noindent where $\kappa$ is the equivariant hyperplane class, the $\lambda_l$ are weights of the $T$-action, and $\hbar$ is the weight of the final $\C^*$-action \cite{coxkatz}.  Let $\mathcal{R}_T H_G^*(N_d)$ denote the localization of $H_G^*(N_d)$ where polynomials in the $\lambda_i$ are inverted.

\begin{definition} A sequence of equivariant classes $Q = \{Q_d\}$ with $Q_d \in \mathcal{R}_T H_G^*(N_d)$ is called \emph{$\Omega$-Euler data} if there is an invertible class $Q_0 = \Omega \in \mathcal{R}_TH_T^*(\pro^n)$ so $Q$ satisfy the relations

\[
i^*_{q_i}(\Omega) \cdot i^*_{p_{i,r}}(Q_d) = i^*_{p_{i,r}}(Q_r) \cdot i^*_{p_{i,0}}(Q_{d-r})
\]

\noindent for each $0 \leq i \leq n, d > 0,$ and $0 < r < d$.
\end{definition}

\noindent We will casually refer to the data as Euler data when $\Omega$ is understood or unimportant.
  
Because pullback to $p_{i,r}$ of a class $\omega \in H_G^*(N_d)$ is given by evaluation in the polynomial ring of the equivariant hyperplane class $\kappa$ at $(\lambda_i + r\hbar)$, it is customary to write $i^*_{p_{i,r}}(\omega)$ as $\omega(\lambda_i + r\hbar)$.  This way, a sequence of equivariant classes $Q$ is $\Omega$-Euler data if

\begin{equation}
\Omega(\lambda_i) \cdot Q_d(\lambda_i + r\hbar) = Q_r(\lambda_i + r\hbar) \cdot Q_{d-r}(\lambda_i).
\label{originaleulercondition}
\end{equation}

We will also need the notion of linked Euler data:

\begin{definition} Two $\Omega$-Euler data $P$ and $Q$ are said to be \emph{linked} if $P_d(\lambda_i) = Q_d(\lambda_i)$ at $\hbar = (\lambda_i - \lambda_j)/d$ for all $i$, $j \neq i$, and $d$.
\end{definition}

Geometrically, Euler data are linked if they agree on multiple coverings of coordinate lines in $N_d$.  The following uniqueness lemma for linked Euler data is proved in \cite{lly:mirror1}:

\begin{lemma}
If $Q$ and $P$ are linked $\Omega$-Euler data such that $\text{deg}_{\hbar}(Q(\lambda_i) - P(\lambda_i)) \leq (n+1)d - 2$, then $Q = P$.
\label{llyuniqueness}
\end{lemma}

A key result from \cite{lly:mirror1} is the existence of an equivariant collapsing map

\begin{equation}
\varphi: \overline{M_{0,0}}(\pro^n \times \pro^1, (d,1)) \rightarrow N_d.
\label{collapsingmap}
\end{equation}

We may take the pullback of any multiplicative equivariant characteristic class $b$ of the bundles $V_d$ on $\overline{M_{0,0}}(\pro^n,d)$ via $\pi$ and then pushforward via $\varphi$ to obtain a sequence $\hat Q = \{\hat Q_d \}$ with $\hat Q_d \in H^*_G(N_d)$:

\[
\hat Q_d = \varphi_*(\pi^*(b(V_d))) \in H^*_G(N_d),
\]

\noindent Unless stated otherwise, $b$ will always be the equivariant Euler class in our applications, although the total Chern class is also interesting \cite{lly:mirror1}.  As a special case of the more general Lemma \ref{showeulerdata}, $\hat Q$ is Euler data.  To be more precise, if $V = V^+ \oplus V^-$ is the given concavex bundle, then $\hat Q$ turns out to be $\frac{e(V^+)}{e(V^-)}$-Euler data.

%
%

\section{$(u,v)$-Euler Data}


For any $m \geq 0$, $u + v = m$, and $k_i \geq 0$, let

\begin{align*}
&\vec{k} = (k_1,...,k_m)\\
&\vec{k}' = (k_1,...,k_v,0,...,0)\\
&\vec{k}'' = (0,...,0,k_{v+1},...,k_m)
\end{align*}

\noindent where $m,u,v$, and the $k_i$ are nonnegative integers.  Also let $|\vec{k}| = \sum k_i$.

\begin{definition} A $(u+v+1)$-sequence $Q = \{Q_{d,\vec{k}}\}$ with $Q_{d,\vec{k}} \in \mathcal{R}_T H_G^*(N_d)$ is called \emph{$(u,v,\beta,\Omega)$-Euler data} if the following hold:

\begin{enumerate}
\item $Q_{0,\vec{k}} = \beta^{|\vec{k}|} \Omega$ for some $\beta \in H^*_T(\pro^n)$ and all $\vec{k}$.
\item $Q$ satisfies the \emph{Euler condition}

\begin{equation}
\Omega(\lambda_i) \cdot Q_{d,\vec{k}}(\lambda_i + r\hbar) = Q_{r,\vec{k}''}(\lambda_i + r\hbar) \cdot Q_{d-r,\vec{k}'}(\lambda_i).
\label{eulercondition}
\end{equation}
\end{enumerate}

\noindent We will call the sequence at some fixed $\vec{k}$ the \emph{height $\vec{k}$} data and denote it by $Q_{\vec{k}}$.  As before, we often suppress $\beta$ and $\Omega$ when unimportant, casually referring to $(u,v)$-Euler data.
\end{definition}


\begin{example} Lian-Liu-Yau's Euler data coincides with $(0,0,\beta,\Omega)$-Euler data.  More generally, the height $\vec{0}$ part of any $(u,v)$-Euler data obeys their gluing relation (\ref{originaleulercondition}).  However, height $\vec{k} \neq \vec{0}$ data is not quite $\Omega$-Euler data nor $\beta^{|\vec{k}|}\Omega$-Euler data but rather a deformation in some sense, and we will say that $Q$ \emph{extends} $Q_{\vec{0}}$.  This constitutes a different generalization of Euler data from the one given in \cite{lly:mirror4}.  The motivating idea is that we wish to work within a framework that allows one to tie marked data coming from a concavex bundle to unmarked data rather than study the problem for each number of markings independently.
\end{example}


%

%
%

\subsection{$(u,v)$-Euler data and Concavex Bundles}

We will now show that a concavex bundle naturally induces $(u,v)$-Euler data when defined through restriction to certain subspaces of an appropriate nonlinear moduli space.

Let $M_d^{u,v} \subset \overline{M_{0,u+v}}(\pro^n \times \pro^1, (d,1))$ be the subspace of the Deligne-Mumford compactification of maps $\pro^1 \rightarrow \pro^n \times \pro^1$ of bi-degree $(d,1)$ with $u+v$ marked points such that the first $u$ marked points are sent to $[0,1]$ under projection to the $\pro^1$ factor and the final $v$ marked points are sent to $[1,0]$ under that projection.  We call $M_d^{u,v}$ the \emph{nonlinear model} and again give it a $G$-action induced by multiplication on the image as in (\ref{gaction}).  Note that $M_d^{0,0}$ is the nonlinear moduli $M_d$ used by Lian-Liu-Yau.  The natural projection to $\overline{M_{0,u+v}}(\pro^n,d)$ will again be denoted by $\pi$, suppressing the $u,v$ dependence:

\[
\pi: M_d^{u,v} \longrightarrow \overline{M_{0,u+v}}(\pro^n,d)
\]

As in the unmarked case, there exists an equivariant collapsing map to $N_d$, still to be denoted $\varphi: M_d^{u,v} \rightarrow N_d$, defined as composition of the Lian-Liu-Yau collapsing map (\ref{collapsingmap}) from $\overline{M_{0,0}}(\pro^n \times \pro^1, (d,1))$ to $N_d$ with the forgetful morphism(s) from $M_d^{u,v}$ to $\overline{M_{0,0}}(\pro^n \times \pro^1, (d,1))$ forgetting the markings.

Fix $u+v=m$.  For $\vec{k} = (k_1,...,k_m)$, let

\[
\chi_{d,\vec{k}}^{V} = \pi^*\left(b(V_d) \prod_{j=1}^mev_1^*(p^{k_j})\right) \in H^*_G(M_d^{u,v}),
\]

\noindent where $p$ is the equivariant hyperplane class on $\pro^n$ and $b$ is any multiplicative equivariant characteristic class.  Also let $\hat Q = \{\hat Q_{d,\vec{k}}\}$:

\[
\hat Q_{d,\vec{k}} = \varphi_*(\chi_{d,\vec{k}}^{V}) \in H^*_G(N_d).
\]

\noindent The $m=0$ case is precisely the Euler data $\hat Q$ studied in \cite{lly:mirror1}.  Our immediate goal is to prove the following:

\begin{lemma}
$\hat Q$ is $(u,v,p,\Omega)$-Euler data, where $\Omega = \frac{e(V^+)}{e(V^-)}$ and $p$ is the equivariant hyperplane class on $\pro^n$.
\label{showeulerdata}
\end{lemma}

We first give some setup and prove a crucial gluing lemma (Lemma \ref{gluinglemma}).
Although the obstruction bundles $V_d$ on $\overline{M_{0,m}}(\pro^n,d)$ induced by a concavex bundle $V = V^+ \oplus V^-$ on $\pro^n$ exist on different spaces, there are nontrivial relations on multiplicative equivariant characteristic classes connecting them.  For any $0 < r < d$, define $\overline M^{i,j}(r,d-r)$ by the pullback diagram

\begin{diagram}
&\overline M^{i,j}(r,d-r)	&\rTo^{p_1}&	&\overline{M_{0,i+1}}(\pro^n,r)&\\
&\dTo^{p_2}	&&	&\dTo_{ev_{i+1}}&\\
&\overline{M_{0,j+1}}(\pro^n,d-r)&	&\rTo^{ev_{j+1}}	&\pro^n&\\
\end{diagram}

\noindent where $p_1$ and $p_2$ are projections and $ev_k$ evaluation at the $k$-th marked point.  We think of $\overline M^{i,j}(r,d-r)$ as consisting of pairs of an $(i+1)$-pointed stable map of degree $r$ and a $(j+1)$-pointed stable map of degree $d-r$, $((f_1,C_1,w_1,...,w_{i+1}),(f_2,C_2,z_1,...,z_{j+1}))$, such that $f_1(w_{i+1}) = f_2(z_{j+1}) = q$.  We also have the diagram

\begin{diagram}
&\overline M^{i,j}(r,d-r) &\rTo^{\psi}	&\overline{M_{0,i+j}}(\pro^n,d)\\
&\dTo^{ev}	&	&\\
&\pro^n	&	&\\
\end{diagram}

\noindent where the image of $ev$ is $q$, and the image of $\psi$ is the stabilization of the map formed by connecting $C_1$ at $w_{i+1}$ with $C_2$ at $z_{j+1}$ of degree $d$ given by $f|_{C_l} = f_l$.

In the case that $r=0$ and $i=0$ or $1$, $\overline{M_{0,i+1}}(\pro^n,r)$ is empty.  To remedy this we replace $\overline{M_{0,i+1}}(\pro^n,0)$ with a copy of $\pro^n$, and replace $ev_{i+1}$ with $id_{\pro^n}$.  This way, $\overline M^{i,j}(0,d) \cong \overline{M_{0,j+1}}(\pro^n,d)$.  We let $\psi$ forget the $(j+1)$-th point when $i=0$ and be the identity when $i=1$.  This way $\psi$ results in an $(i+j)$-point curve in all situations.  The case $r=d$ is similarly modified when $j=0$ or $1$.

\begin{lemma}[Gluing Lemma]
If $b$ is a multiplicative equivariant characteristic class then the following relation on $\overline M^{i,j}(r,d-r)$ holds for $0<r<d$:
\[
ev^*\Omega \cdot \psi^*b(V_d) = p_1^*ev_{i+1}^*b(V_r) \cdot p_2^*ev_{j+1}^*b(V_{d-r})
\]

\noindent where $\Omega = \frac{b(V^+)}{b(V^-)}$.
\label{gluinglemma}
\end{lemma}

\begin{proof}

Suppose $((f_1,C_1,w_1,...,w_{i+1}),(f_2,C_2,z_1,...,z_{j+1}))$ map to $(f,C,w_1,...,w_i,z_1,...,z_j)$ under $\psi$.  This gives rise to the exact sequence

\[
0 \rightarrow f^*V \rightarrow f_1^*V \oplus f_2^*V \rightarrow V|_q \rightarrow 0
\]

\noindent In the case that $V$ is convex, passing to cohomology gives the exact sequence

\[
0 \rightarrow H^0(C,f^*V) \rightarrow H^0(f_1^*V) \oplus H^0(f_2^*V) \rightarrow V|_q \rightarrow 0,
\]

\noindent from which the lemma follows easily.  Likewise, if $V$ is concave we have the exact sequence

\[
0 \rightarrow V|_q \rightarrow H^1(C,f^*V) \rightarrow H^1(f_1^*V) \oplus H^1(f_2^*V) \rightarrow 0,
\]

\noindent and again the corresponding gluing lemma follows.  The general concavex case is then an easy combination of the two cases. \qed

\end{proof}

We now prove Lemma \ref{showeulerdata}:

\begin{proof}
First observe that $\hat Q_{d,\vec{k}}$ is automatically polynomial in $\hbar$, as it is an equivariant class on $N_d$.  We show that $\hat Q$ satisfies ($\ref{eulercondition}$) by pulling the calculation back to $M_d$:

\begin{equation}
\hat Q_{d,\vec{k}}(\lambda_i + r\hbar) = \int_{(N_d)_G} \phi_{p_{i,r}} \hat Q_{d,\vec{k}} = \int_{(M_d^{u,v})_G} \varphi^*(\phi_{p_{i,r}})\chi_{d,\vec{k}}^{V},
\label{pullback}
\end{equation}

\noindent where $\phi_{p_{i,r}} = \prod_{(j,s) \neq (i,r)}(\kappa - (\lambda_j + s\hbar))$ is the equivariant Poincar\'e dual to the fixed point $p_{i,r}$.  We wish to calculate (\ref{pullback}) via localization and so must understand the $G$-fixed components.  Each fixed component $M_{\Gamma} \subseteq M_d$ is labeled by a decorated graph $\Gamma$ in the usual way \cite{coxkatz} \cite{mirrorsymmetry}, although we will not directly need this description.  Furthermore, we must have  $\varphi(M_{\Gamma}) = p_{i,r}$.

We first consider the case $0 < r < d$.  To obtain a better description of these $G$-fixed components and in particular their normal bundles inside of $M_d^{u,v}$, we build a variation $\overline{\psi}$ of the map $\psi$ given above as follows: for each point of $\overline M^{u,v}(r,d-r)$ represented by the pair $((f_1,C_1,w_1,...,w_{u+1}),(f_2,C_2,z_1,...,z_{v+1}))$ with $f_1(w_{u+1}) = q = f_2(z_{v+1})$, let $C = C_0 \cup C_1 \cup C_2$ where $C_0 \cong \pro^1$, with $[1,0] \in C_0$ glued to $w_{u+1} \in C_1$, and $[0,1] \in C_0$ glued to $z_{v+1} \in C_2$.  Then define $f: C \rightarrow \pro^1 \times \pro^n$ by

\begin{align*}
f|_{C_0}(z) &= (z,q)\\
f|_{C_1}(z) &= ([1,0],f_1(z))\\
f|_{C_2}(z) &= ([0,1],f_2(z))\\
\end{align*}

\noindent This induces a map $\overline M^{u,v}(r,d-r) \rightarrow \overline{M_{0,u+v}}(\pro^n \times \pro^1,(d,1))$.  Moreover, the image is contained in $M_d^{u,v}$ and so defines a map $\bar \psi$

\[
\bar \psi: \overline{M}^{u,v}(r,d-r) \hookrightarrow M_d^{u,v}.
\]

Fix $0 \leq i \leq n$ for the moment and let $\{ F_k^j \}$ denote the set of $T$-fixed components of $\overline{M_{0,j}}(\pro^n,k)$ such that the $j$-th marked point is sent to $q_i \in \pro^n$.  In the present case, it is clear that $F_r^1 \times F_{d-r}^2 \subseteq \overline M (r,d-r)$, and we identify this subset with its image under $\bar \psi$, writing $F_r^1 \times F_{d-r}^2 \subseteq M_d^{u,v}$.  It is not hard to see that all the fixed components $M_{\Gamma}$ of $M_d^{u,v}$ with $\varphi(M_{\Gamma}) = p_{i,r}$ arise in this way.  Hence by functorial localization, (\ref{pullback}) becomes

\[
\hat Q_{d,\vec{k}}(\lambda_i + r\hbar) = \sum_{\{F_r^{u+1} \times F_{d-r}^{v+1}\}} \int_{(F_r^{u+1} \times F_{d-r}^{v+1})_G} \frac{i^*_{\Gamma}(\varphi^*(\phi_{p_{i,r}})\chi_{d,\vec{k}}^{V})}{\text{Euler}_G(N_{F_r^{u+1} \times F_{d-r}^{v+1}})}.
\]

We want to identify the class $\text{Euler}_G(N(F_r^{u+1} \times F_{d-r}^{v+1}))$ in the equivariant K-theory.  By considering the various inclusions at hand it is straightforward to find

\[
N(F_r^{u+1} \times F_{d-r}^{v+1}) = N(F_r^{u+1}) + N(F_{d-r}^{v+1}) - 2 T_{q_i}\pro^n + T_{q_i}\pro^n + N(\bar \psi) - N(J),
\]

\noindent where $J$ is the inclusion map $J: M_d^{u,v} \hookrightarrow \overline{M_{0,u+v}}(\pro^n \times \pro, (d,1))$.  Thus, omitting pullbacks for simplicity,

\[
\text{Euler}_G(N(F_r^{u+1} \times F_{d-r}^{v+1})) = \text{Euler}_G(N(\bar \psi)) \frac{\text{Euler}_T(N(F_r^{u+1}))\text{Euler}_T(N(F_{d-r}^{v+1}))}{\text{Euler}_G(N(J))\text{Euler}_T(T_{q_i}\pro^n)}.
\]

\noindent These pieces may be identified as

\begin{align*}
\text{Euler}_T(T_{q_i}\pro^n) &= \prod_{j \neq i} (\lambda_i - \lambda_j)\\
\text{Euler}_G(N(J))) &= (-1)^{v}\hbar^{u+v}\\
\text{Euler}_G(N(\bar \psi)) &= \frac{-\hbar^{-2}}{(-\hbar - c_1^G(\mathcal{L}_{d-r,v+1}))(\hbar - c_1^G(\mathcal{L}_{r,u+1}))}
\end{align*}

\noindent where $\mathcal{L}_{r,j}$ is the line bundle on $\overline{M_{0,j}}(\pro^n,d)$ given by the cotangent line to the $j$-th marked point on each curve.  Furthermore, by Lemma \ref{gluinglemma}, on $F^{u+1}_r \times F^{v+1}_{d-r}$

\[
\Omega(\lambda_i) b(V_d) = b(V_r)b(V_{d-r})
\]

\noindent Putting all this together, we have:

\begin{align}
\hat Q_{d,\vec{k}}(\lambda_i + r\hbar) =& \int_{(M_d^{u,v})_G} \varphi^*(\phi_{i,r})\chi_{d,\vec{k}}^{V} \notag\\
	=& \enspace (-1)^v \hbar^{u+v-2}(\Omega(\lambda_i))^{-1}i^*_{p_{i,r}}(\phi_{i,r}) \prod_{j \neq i} (\lambda_i - \lambda_j) \times \notag\\
	& \sum_{\{F_r^{u+1}\}}	\int_{(F_r^{u+1})_T} \frac{\pi^*_{u+1}(b(V_r)\prod_{j=1}^{u} ev_{j}^*p^{k_{j+v}})}{\text{Euler}_T(N(F_r^{u+1}))(\hbar - c_1^G(\mathcal{L}_{r,u+1}))}\times \notag\\
	&	\sum_{\{F_{d-r}^{v+1} \}} \int_{(F_{d-r}^{v+1})_T} \frac{\pi_{v+1}^*(b(V_{d-r}) \prod_{j=1}^{v}ev_j^*p^{k_j})}{\text{Euler}_T(N(F_{d-r}^{v+1}))(\hbar + c_1^G(\mathcal{L}_{d-r,v+1}))}.\label{qhatfull}
\end{align}

\noindent A similar analysis for $r=0$ yields

\begin{equation}
\hat Q_{d,\vec{k}}(\lambda_i) = (-1)^v i^*_{p_{i,0}}(\phi_{i,0}) \hbar^{v-1} \prod_{j=1}^u \lambda_i^{k_{v+j}} \sum_{\{F_d^{v+1}\}} \int_{(F_d^{v+1})_T} \frac{\pi_{v+1}^*(b(V_{d}) \prod_{j=1}^{v}ev_j^*p^{k_j})}{\text{Euler}_G(N(F_d^{v+1}))(\hbar + c_1^G(\mathcal{L}_{d,v+1}))}\\
\label{qhat}
\end{equation}

\noindent so that

\begin{equation}
\hat Q_{d-r,\vec{k}'}(\lambda_i) = (-1)^v i^*_{p_{i,0}}(\phi_{i,0}) \hbar^{v-1} \sum_{\{F_{d-r}^{v+1}\}} \int_{(F_{d-r}^{v+1})_T} \frac{\pi_{v+1}^*(b(V_{d-r}) \prod_{j=1}^{v}ev_j^*p^{k_j})}{\text{Euler}_G(N(F_{d-r}^{v+1}))(\hbar + c_1^G(\mathcal{L}_{d-r,v+1}))}.\\
\label{qhatdminusr}
\end{equation}

\noindent Likewise, for $r=d$ we find

\begin{equation}
\hat Q_{r,\vec{k}''}(\lambda_i+r\hbar) = i^*_{p_{i,r}}(\phi_{i,r}) \hbar^{u-1} \sum_{\{F_r^{u+1}\}}	\int_{(F_r^{u+1})_T} \frac{\pi^*_{u+1}(b(V_r)\prod_{j=1}^{u}ev^*_jp^{k_{j+v}})}{\text{Euler}_G(N(F_r^{u+1}))(\hbar - c_1^G(\mathcal{L}_{r,u+1}))}\\,
\label{qhat2}
\end{equation}

\noindent Combining (\ref{qhatfull}), (\ref{qhatdminusr}), and (\ref{qhat2}), it is then straightforward to check that that $\hat Q$ satisfies (\ref{eulercondition}) and thus is $(u,v,p,\Omega)$-Euler data.  (Note that in (\ref{qhatfull}), $i^*_{p_{i,r}}(\phi_{i,r})$ is on $N_d$ but in (\ref{qhatdminusr}) $i^*_{p_{i,0}}(\phi_{i,0})$ is on $N_r$ and in (\ref{qhat2}) $i^*_{p_{i,r}}(\phi_{i,r})$ on $N_{d-r}$.) \qed
\end{proof}

\subsection{Uniqueness for $(0,1)$-Euler data}

In this section we will restrict ourselves to studying a uniqueness lemma for $(0,1)$-Euler data.  The discussion has a natural analog for $(1,0)$-Euler data, which we will not state, that will only be used briefly in the discussion on two markings.

For notational simplicity, we will refer to the height using $k$ rather than $\vec{k} = (k_1)$.

\begin{lemma}[Uniqueness Lemma]
If $Q = \{Q_{d,k}\}$ and $P = \{P_{d,k}\}$ are $(0,1,\beta,\Omega)$-Euler data such that $Q_0 = P_0$ and deg$_{\hbar}(Q_{d,k}(\lambda_i) - P_{d,k}(\lambda_i)) \leq (n+1)d - 1$ for all $i$, $k$, and $d$, then $Q = P$.
\label{uniquenesslemma}
\end{lemma}

\begin{proof}
We will prove $Q_{d,k} = P_{d,k}$ for each $k$ by induction on $d$.  Notice that $Q_{0,k} = \beta^k\Omega = P_{0,k}$ by assumption.  Suppose that for any fixed $k > 0$, $Q_{d',k} = P_{d',k}$ for $0 < d' < d$.  Since the intersection pairing is nondegenerate, it will be sufficient to show that

\[
\int_{(N_d)_G} \kappa^s (Q_{d,k}- P_{d,k}) = 0 \qquad \text{for all } s \geq 0,
\]

\noindent where $\kappa$ is the equivariant hyperplane class on $N_d$.  We compute this integral by localization:

\begin{align*}
\int_{(N_d)_G} \kappa^s (Q_{d,k} - P_{d,k}) &= \sum_{0 \leq i \leq n} \sum_{0 \leq r \leq d} \frac{(\lambda_i+r\hbar)^s i^*_{p_{i,r}}(Q_{d,k}-P_{d,k})}{\text{Euler}_G(N_{p_{i,r}})}\\
	&= \sum_{i=0}^n \left[ \frac{\lambda_i^s (Q_{d,k}(\lambda_i)-P_{d,k}(\lambda_i))}{\text{Euler}_G(N_{p_{i,0}})} + \frac{(\lambda_i + d\hbar)^s (Q_{d,k}(\lambda_i+d\hbar)-P_{d,k}(\lambda_i+d\hbar))}{\text{Euler}_G(N_{p_{i,d}})} \right]\\
	&= \sum_{i=0}^n \frac{\lambda_i^s (Q_{d,k}(\lambda_i)-P_{d,k}(\lambda_i))}{\text{Euler}_G(N_{p_{i,0}})}\\
\end{align*}

The second line follows from the first because $Q$ satisfies the Euler condition, allowing $Q_{d,k}(\lambda_i + r\hbar)$ to be expressed in terms of $Q_{1,k}, \dots, Q_{d-1,k}$ and $Q_{r,0}$ whenever $0 < r < d$, and likewise for $P_{d,k}(\lambda_i+r\hbar)$.  By the inductive hypothesis these agree, and the only terms remaining in the difference are the $r=0$ and $r=d$ cases.  Furthermore, using the assumption $Q_0 = P_0$ when $r=d$,

\begin{align*}
Q_{d,k}(\lambda_i + d\hbar) &= \Omega(\lambda_i)^{-1} Q_{d,0}(\lambda_i + d\hbar) Q_{0,k}(\lambda_i)\\
	&=\Omega(\lambda_i)^{-1} P_{d,0}(\lambda_i + d\hbar) P_{0,k}(\lambda_i)\\
	&=P_{d,k}(\lambda_i + d\hbar), \\
\end{align*}

\noindent so this difference vanishes as well giving the final line.  Computing the normal bundle in the final line, we find that 

\begin{equation}
\int_{(N_d)_G} \kappa^s (Q_{d,k} - P_{d,k}) = \sum_{i=0}^n \frac{\lambda_i^s A_{i}(\hbar)}{\hbar^d}
\label{difference}
\end{equation}

\noindent where

\[
A_i(\hbar) = \frac{(-1)^d}{d!\prod_{j \neq i}(\lambda_i - \lambda_j)} \frac{Q_{d,k}(\lambda_i) - P_{d,k}(\lambda_i)}{\prod_{j \neq i} \prod_{m=1}^d(\lambda_i - (\lambda_j + m\hbar))}.
\]

We claim that $A_i$ is in fact polynomial in $\hbar$.  It is easy to see that $(0,1,\beta,\Omega)$-Euler data automatically satisfies the \emph{self-linking condition} that $Q_{d,k}(\lambda_i) = \beta(\lambda_j)^k Q_{d,0}(\lambda_i)$ at $\hbar = (\lambda_i - \lambda_j)/d$ for all $i$, $j \neq i$, $k$, and $d$.  It follows that at $\hbar = (\lambda_i - \lambda_j)/d$

\[
Q_{d,k}(\lambda_i) - P_{d,k}(\lambda_i) = \beta(\lambda_j)^k(Q_{d,0}(\lambda_i) - P_{d,0}(\lambda_i)) = 0,
\]

\noindent which cancels out some zeros of the denominator.  The other zeros at $\hbar = (\lambda_i - \lambda_j)/s$ where $0 < s < d$ cancel as well again by repeated use of the self-linking condition (geometrically if $Q_{d,k}$ and $Q_{s,k}$ are both multiples of $Q_{d,0}$ on covers of coordinate lines then they are multiples of each other).  We have shown that $A_i$ is indeed a polynomial in $\hbar$.  However, the left side of (\ref{difference}) is naturally a polynomial in $\hbar$.  The degree bound implies that

\[
\text{deg}_{\hbar} A_i \leq (n+1)d - 1 - nd = d-1
\]

\noindent Since $\hbar^d$ divides $A_i$ we conclude that $A_i$, and hence the left hand side of (\ref{difference}), always vanish.  The lemma then follows. \qed
\end{proof}

\section{Hypergeometric Data}

Throughout this section and the next, we will take $\beta \in H_T^*(\pro^n)$ to be $p$, the equivariant hyperplane class.  The goal is first to extend $(0,0)$-Euler data to $(0,1)$-Euler data in a natural way, and then to transform $(0,1)$-Euler data into ``linked" data so certain hypergeometric series associated to the data obey simple relations.  These notions will be made precise shortly.  Ultimately we prove Theorem \ref{maintheorem} showing that this can be done in a way that controls the growth of $\hbar$.

\subsection{An Extension Lemma}

Let $\mathcal{S}_0$ denote the set of all sequences $B = \{B_d\}$ such that $B_d \in H_T^*(\pro^n)(\hbar)$ for all $d$, where $\hbar$ is now a formal parameter.  The hypergeometric function associated to $B \in \mathcal{S}_0$ is a formal function taking values in $H^*_T(\pro^n) \otimes \C (\hbar)$ given by

\[
HG[B](t) = e^{-pt/\hbar}\left[ B_0 + \sum_{d = 1}^\infty e^{dt} \frac{B_d}{\prod_{l=0}^n \prod_{m=1}^d (p - \lambda_l - m\hbar)} \right].
\]

%

Let $I_d: \pro^n \hookrightarrow N_d$ be the equivariant embedding given by

\[
[a_0,\dots,a_n] \mapsto [a_0 w_1^d,\dots,a_n w_1^d].
\]

\noindent Also let $\mathcal{S}$ be the set of all sequences $Q = \{Q_d\}$ of equivariant classes with $Q_d \in \mathcal{R}_T H_G^*(N_d)$.  We then have a natural map $\mathcal{I}: \mathcal{S} \rightarrow \mathcal{S}_0$, defined via $(\mathcal{I}(Q))_d = I_d^*(Q_d)$ for every $d$.  With this notation, the hypergeometric series associated to  the height $\vec{k}$ sequence in $(u,v,p,\Omega)$-Euler data is given by 

\begin{equation}
\label{hyperdef}
HG[\mathcal{I}(Q_{\vec{k}})](t) = e^{-pt/\hbar}\left[ \Omega p^k  + \sum_{d = 1}^\infty e^{dt} \frac{I_d^*(Q_{d,\vec{k}})}{\prod_{l=0}^n \prod_{m=1}^d (p - \lambda_l - m\hbar)} \right].
\end{equation}

For the remainder of this section we will deal only with $(0,1)$-Euler data, and again adopt the notation $Q_k$ in place of $Q_{\vec{k}}$ for the data at a particular height.

\begin{lemma}[Extension Lemma]
There is a map $\rho: \mathcal{S} \rightarrow \mathcal{S}$ such that $\rho$ gives a natural extension of $(0,0,-,\Omega)$-Euler data $Q_0 \in \mathcal{S}$ to $(0,1,p,\Omega)$-Euler data $Q$ in such a way that

\[
Q_k = \rho^k(Q_0)
\]

\noindent for all $k > 0$ and with the property that 

\begin{equation}
HG[\mathcal{I}(Q_k)](t) = \left \{ -\hbar \frac{\partial}{\partial t} \right \}^k HG[\mathcal{I}(Q_0)](t).
\label{hyperdiff}
\end{equation}
\label{extensionlemma}
\end{lemma}

\begin{proof}
Given $Q_0 = \{Q_{0,k}\}$, define $\rho^k(Q_0) = \{Q_{d,k}\}$ via

\[
Q_{d,k}(\lambda_i + r\hbar) = (\lambda_i - (d-r)\hbar)^k Q_{d,0}(\lambda_i + r\hbar).
\]

\noindent That this extension satisfies \ref{eulercondition} is easily checked, and as the restriction at each fixed point of $N_d$ is polynomial in $\hbar$ so is the class $Q_{d,k}$ itself.  Furthermore, $i^*_{q_i}I_d^*(Q_{d,k}) = Q_{d,k}(\lambda_i) = (\lambda_i - d\hbar)^k Q_{d,0}(\lambda_i)$ for every fixed point $q_i \in \pro^n$, from which we conclude

\[
I_d^*(Q_{d,k}) = (p-d\hbar)^kI_d^*(Q_{d,0}).
\]

\noindent Equation (\ref{hyperdiff}) is now verified by straightforward power series manipulation. \qed

\end{proof}

\subsection{Mirror Transformations}

Lemma \ref{extensionlemma} implies we may always extend some initial $(0,0,\Omega)$-Euler data to $(0,1,p,\Omega)$-Euler data in a simple way.  However, hypergeometric data associated to the extension will behave badly in that it grows in powers of $\hbar$ as the height $k$ increases.  It turns out the correct way to control the hypergeometric data is transform Euler data into new ``linked" data.  More generally, one may wish to manipulate $(u,v)$-Euler data.

We now give some natural extensions to definitions found in \cite{lly:mirror1}:

\begin{definition} Two $(u,v,\beta,\Omega)$-Euler data $P$ and $Q$ are called \emph{linked} if $Q_0$ and $P_0$ are linked $(0,0,\Omega)$-Euler data, meaning that $Q_{d,0}(\lambda_i)$ and $P_{d,0}(\lambda_i)$ agree at $\hbar = (\lambda_i - \lambda_j)/d$ for all $i,j$, and $d$.
\end{definition}


Let $\mathcal{A}^{\Omega}$ denote the set of all $(u,v,p,\Omega)$-Euler data.


\begin{definition} A \emph{mirror transformation} on $(u,v,\beta,\Omega)$-Euler data is an invertible map $\mu: \mathcal{A}^{\Omega} \rightarrow \mathcal{A}^{\Omega}$ so that $\mu(Q)$ is linked to $Q$ for all $Q \in \mathcal{A}^{\Omega}$.  We will sometimes deal with mirror transformations that alter the Euler data at one particular height $\vec{k}$ while fixing all other heights; in this case, we call $\mu$ a \emph{mirror transformation of height $\vec{k}$}.  Clearly the composition of mirror transformations of different heights will be a mirror transformation.
\end{definition}

Lemmas \ref{trans1}, \ref{trans2}, and \ref{trans3} below describe specific mirror transformations on $(u,v,p,\Omega)$-Euler data that will be needed to establish Theorem \ref{maintheorem} or in the examples.  The existence of a mirror transformation described at each height by a relation on hypergeometric data is established by first showing that there is a well-defined image sequence in $\mathcal{S}_0$ satisfying the relation, then that this sequence lifts well to a sequence in $\mathcal{S}$ so the lift satisfies the Euler condition and is polynomial in $\hbar$, and of course is linked to the original data.

If one ignores the higher height data, Lemmas \ref{trans1} and \ref{trans2} exactly reduce to Lemma 2.15 of \cite{lly:mirror1}.  Their proofs are essentially the same argument with small modifications; we produce those arguments for Lemma \ref{trans1}; Lemma \ref{trans2} is simpler and left as an exercise.

\begin{lemma}
There exists a mirror transformation $\mu: \mathcal{A}^{\Omega} \rightarrow \mathcal{A}^{\Omega}$ so that for any $g \in e^t\mathcal{R}_T[[e^t]]$,

\[
HG[\mathcal{I}(\mu(Q)_{\vec{k}})](t) = HG[\mathcal{I}(Q_{\vec{k}})](t + g)
\]
\noindent for all $Q \in \mathcal{A}^{\Omega}$ and \emph{all} heights ${\vec{k}}$.
\label{trans1}
\end{lemma}

\begin{lemma}
There exists a mirror transformation $\nu: \mathcal{A}^{\Omega} \rightarrow \mathcal{A}^{\Omega}$ of \emph{arbitrary} height ${\vec{k}}$ such that for any $f \in e^t\mathcal{R}_T[[e^t]]$,

\[
HG[\mathcal{I}(\nu(Q)_{\vec{k}})](t) = e^{f / \hbar}HG[\mathcal{I}(Q_{\vec{k}})](t).
\]

\label{trans2}
\end{lemma}

\begin{proof}
To prove Lemma \ref{trans1}, let $\mathcal{I}(Q_{\vec{k}}) = B_{\vec{k}} \in \mathcal{S}_0$ so that

\[
HG[B_{\vec{k}}](t + g) = e^{-pt/\hbar}e^{-pg/\hbar} \sum_{d = 0}^\infty e^{dt}e^{dg} \frac{B_{d,{\vec{k}}}}{\prod_{l=0}^n \prod_{m=1}^d (p - \lambda_l - m\hbar)}.
\]

Expand $e^{dg} = \sum_{s \geq 0} g_{d,s} e^{st}$ and $e^{-pg/\hbar} = \sum_{s \geq 0} g'_s e^{st}$, where $g_{d,s} \in \mathcal{R}_T$ and $g'_{s} \in \mathcal{R}_T[p/\hbar]$.  It is straightforward to find that 

\[
\tilde B_{d,{\vec{k}}} = B_{d,{\vec{k}}}' + \sum_{r=0}^{d-1} g_{d-r}'B_{r,{\vec{k}}}'\prod_{l=0}^n \prod_{m=r+1}^d (p - \lambda_l - m\hbar),
\]

\noindent where

\[
B_{d,{\vec{k}}}' = B_{d,{\vec{k}}} + \sum_{r=0}^{d-1} g_{r,d-r}B_{r,{\vec{k}}}\prod_{l=0}^n \prod_{m=r+1}^d (p - \lambda_l - m\hbar),
\]

\noindent defines a unique $\tilde B_{\vec{k}} = \{\tilde B_{d,{\vec{k}}}\}$ such that 

\[
HG[\tilde B_{\vec{k}}](t) = HG[B_{\vec{k}}](t + g).
\]

For each ${\vec{k}}$, we want to lift $\tilde B_{\vec{k}}$ to a sequence $\tilde Q_{\vec{k}} = \{ \tilde Q_{d,{\vec{k}}} \}$ with $\tilde Q_{d,{\vec{k}}} \in \mathcal{R}_T H_G^*(N_d)$ in such a way that $\tilde Q_{\vec{k}}$ is the height ${\vec{k}}$ data in $(u,v,p,\Omega)$-Euler data $\tilde Q$ linked to $Q$.  We define the \emph{Lagrange map} $\mathcal{L}$ by

\begin{equation}
(\mathcal{L}(B_{\vec{k}}))_d(\lambda_i + r\hbar) = \Omega(\lambda_i)^{-1} \overline{B_{r,0}(\lambda_i)} B_{d-r,{\vec{k}}}(\lambda_i)
\label{lagrangelifting}
\end{equation}

\noindent for each $k \geq 0$.  Here overline denotes the conjugate map on $\pro^n$ sending $\hbar$ to $-\hbar$ (and more generally on $N_d$) studied in \cite{lly:mirror1}.  The relevant property of the conjugate map here is that $Q_{d,0}(\lambda_i + d\hbar) = \overline{Q_{d,0}(\lambda_i)}$, from which one sees that the image of the Lagrange map satisfies the Euler condition \ref{eulercondition}.  One may check that $\mathcal{I} \circ \mathcal{L} = \text{id}_{\mathcal{S}_0}$, and that $\mathcal{L} \circ \mathcal{I} = \text{id}_{\mathcal{S}}$ when restricted to Euler data.  Define $\tilde Q_k$ to be $\mathcal{L}(\tilde B_k)$.  It is clear from the definition of $\tilde B_{d,\vec{k}}$ that $\tilde Q_{d,0}(\lambda_i) = \tilde B_{d,0}(\lambda_i)$ agrees with $Q_{d,0}(\lambda_i) = B_{d,0}(\lambda_i)$ at $\hbar = (\lambda_i - \lambda_j)/d$, so $\tilde Q$ is linked to $Q$.  We now need only to show that $\tilde Q_{d,k} \in \mathcal{R}_T H_G^*(N_d)$.

Broadly, if $Q$ is $(u,v,p,\Omega)$-Euler data, then multiplying both sides of (\ref{lagrangelifting}) by the identity

\begin{align*}
e^{d\tau} \frac{e^{(\lambda_i + r \hbar)(t-\tau)/\hbar}}{\prod_{j=0}^n \prod_{m=0, (j,m) \neq (i,r)}^d (\lambda_i + r\hbar - \lambda_j - m\hbar)} =& \frac{1}{\prod_{j \neq i}(\lambda_i - \lambda_j)} \times \frac{e^{rt}}{\prod_{j=0}^n \prod_{m=1}^r (\lambda_i - \lambda_j + m\hbar)}\\
& \times e^{\lambda_i (t-\tau)/\hbar} \frac{e^{(d-r)\tau}}{\prod_{j=0}^n \prod_{m=1}^r (\lambda_i - \lambda_j - m\hbar)}
\end{align*}

\noindent and summing over $0 \leq i \leq n$, $0 \leq r \leq d$, and then finally over $d = 0$ to $\infty$ yields (via localization) the identity

\begin{equation}
\sum_{d \geq 0} e^{d\tau} \int_{(N_d)_G} e^{\kappa (t-\tau)/\hbar} Q_{d,\vec{k}} = \int_{(\pro^n)_T} (\Omega^{-1} \overline{HG[B_0](t)} HG[B_{\vec{k}}](\tau))
\label{221analog}
\end{equation}

\noindent with $B_0 = \mathcal{I}(Q_0)$ and $B_{\vec{k}} = \mathcal{I}(Q_{\vec{k}})$.  Equation (\ref{221analog}) is generally useful in establishing that lifts under the Lagrange map are polynomial in $\hbar$.

In the present case, \ref{221analog} gives

\begin{align*}
\sum_{d \geq 0} e^{d\tau} &\int_{(N_d)_G} e^{\kappa(t - \tau)/\hbar} \tilde Q_{d,\vec{k}} \\
&= \int_{(\pro^n)_T} \Omega^{-1} \overline{HG[\mathcal{I}(\mu(Q)_0)](t)} HG[\mathcal{I}(\mu(Q)_{\vec{k}})](\tau)\\
&= \int_{(\pro^n)_T} \Omega^{-1} \overline{HG[\mathcal{I}(Q_0)](t + g(e^t))} HG[\mathcal{I}(Q_{\vec{k}})](\tau + g(e^{\tau})) \\
&= \sum_{d \geq 0} e^{d\tau}  \int_{(N_d)_G} e^{\kappa(t + \bar g(e^t) - \tau - g(e^{\tau}))/\hbar}Q_{d,\vec{k}}\\
\end{align*}

Decompose $g = g_+ + g_-$, where $\bar g_{\pm} = \pm g_{\pm}$, and set $q = e^{\tau}$ and $\zeta = (t-\tau)/\hbar$, so that $g_+(qe^{\zeta\hbar}) - g_+(q) \in \hbar\mathcal{R}_T[[q, \zeta]]$.  Moreover, $g_-(q)$, $g_-(qe^{\zeta\hbar}) \in \hbar\mathcal{R}_T[[q, \zeta]]$ naturally.  Since $Q_{d,\vec{k}} \in \mathcal{R}_TH^*_G(N_d)$, the final summation above lies in $\mathcal{R}_T[[q, \zeta]]$, and so too much the initial summation.  It follows that 

\[
\int_{(N_d)_G} \kappa^s \tilde Q_{d,\vec{k}} \in \mathcal{R}_T
\]

\noindent for all $s \geq 0$.  A priori $\tilde Q_{d,\vec{k}} = a_N\kappa^N + ... + a_0$, with $a_i \in \mathcal{R}_G$, and pairing this expression with various powers of $\kappa$ and integrating over $(N_d)_G$ gives $a_i \in \mathcal{R}_T$.   Thus $\tilde Q_{d,\vec{k}} \in \mathcal{R}_TH^*_G(N_d)$ is $(u,v)$-Euler data and we may define the mirror transformation $\mu$ via $\mu(Q) = \tilde Q$. 

The proof of Lemma \ref{trans2} is carried out by a similar argument.\qed
\end{proof}

These two lemmas describe mirror transformations that generalize those in \cite{lly:mirror1} used to manipulate the first two terms in the $\hbar^{-1}$-expansion of hypergeometric data so a uniqueness result applies.  We will need a new class of mirror transformations that allow one to further manipulate the degree $0$ term of the $\hbar^{-1}$-expansion of higher derivatives of hypergeometric data in such a way as to eliminate all the purely equivariant terms that appear after applying Lemma \ref{extensionlemma}.  The following lemma provides that new class of transformations:

\begin{lemma}
There exists a mirror transformation $\eta: \mathcal{A}^{\Omega} \rightarrow \mathcal{A}^{\Omega}$ of arbitrary height $\vec{k}$ so that for any $f \in e^t\mathcal{R}_T[[e^t]]$ and $\vec{k}'$, 

\[
HG[\mathcal{I}(\eta(Q)_{\vec{k}})](t) = HG[\mathcal{I}(Q_{\vec{k}})](t) + f \cdot HG[\mathcal{I}(Q_{\vec{k}'})](t)
\]

\noindent for all $Q \in \mathcal{A}^{\Omega}$.
\label{trans3}
\end{lemma}

\begin{proof}
Write $\mathcal{I}(Q_{\vec{k}}) = B_{\vec{k}} = \{ B_{d,\vec{k}} \}$ and likewise $\mathcal{I}(Q_{\vec{k}'}) = B_{\vec{k}'} = \{B_{d,\vec{k}'} \}$.  We need to show that there is a sequence $\tilde B_{\vec{k}} = \{ \tilde B_{d,\vec{k}} \} \in \mathcal{S}_0$ so that

\[
HG[\tilde B_{\vec{k}}](t) = HG[B_{\vec{k}}](t) + f \cdot HG[\mathcal{I}(Q_{\vec{k}'})](t)
\]

\noindent and that this sequence lifts well under the Lagrange map to linked $(u,v)$-Euler data.

It is straightforward to calculate that if we write $f(t) = \sum_{s \geq 1} f_s(t) e^{st}$, the desired sequence $\tilde B_k \in \mathcal{S}_0$ is given by

\[
\tilde B_{d,\vec{k}} = B_{d,\vec{k}} + \sum_{s=1}^{d} f_s(t) B_{d-s,\vec{k}'} \prod_{l=0}^n \prod_{m=d-s+1}^d (p - \lambda_l - m\hbar).
\]

\noindent The Lagrange lift $\mathcal{L}(\tilde B_{\vec{k}})$, given as before by 

\[
(\mathcal{L}(\tilde B_{\vec{k}}))_d(\lambda_i + r\hbar) = \Omega(\lambda_i)^{-1} \overline{B_{r,0}(\lambda_i)} \tilde B_{d-r,\vec{k}}(\lambda_i)
\]

\noindent will satisfy the Euler condition (\ref{eulercondition}) by design, and will be linked to the original data as $\tilde B_{d,\vec{k}}(\lambda_i)$ and $B_{d,\vec{k}}(\lambda_i)$ agree at $\hbar = (\lambda_i - \lambda_j)/d$ for all $i,j$, and $d$.  Define $\eta(Q)_{\vec{k}} = \mathcal{L}(B_{\vec{k}})$.  We need to check that that $\nu(Q)_{\vec{k}}$ is in fact in $\mathcal{R}_T H_G^*(N_d)$.  Applying (\ref{221analog}), we have 

\begin{align*}
\sum_{d \geq 0} e^{d\tau} &\int_{(N_d)_G} e^{\kappa(t - \tau)/\hbar} (\eta(Q)_{d,\vec{k}}) \\
&= \int_{(\pro^n)_T} \Omega^{-1} \overline{HG[\mathcal{I}(\nu(Q)_0)](t)} HG[\mathcal{I}(\eta(Q)_{\vec{k}})](\tau)\\
&= \int_{(\pro^n)_T} \Omega^{-1} \overline{HG[\mathcal{I}(Q_0)](t)} \left( HG[\mathcal{I}(Q_{\vec{k}})](\tau) + f \cdot HG[\mathcal{I}(Q_{\vec{k}'})](\tau) \right)\\
&= \sum_{d \geq 0} e^{d\tau}  \int_{(N_d)_G} e^{\kappa(t - \tau)/\hbar} \left[ Q_{d,\vec{k}} + f \cdot Q_{d,\vec{k}'} \right] \\
\end{align*}

\noindent Both $Q_{d,\vec{k}}$ and $Q_{d,\vec{k}'}$ are in $\mathcal{R}_TH^*_G(N_d)$ as $Q$ is Euler data.  An argument similar to that of Lemma \ref{trans1} shows that $\nu(Q)_{d,\vec{k}}$ must also be in $\mathcal{R}_TH^*_G(N_d)$.
\qed

\end{proof}

\noindent We now have all the ingredients necessary to prove:

\begin{theorem}  Suppose that $P_0$ is $(0,0,\Omega)$-Euler data whose hypergeometric data has series expansion in $\hbar^{-1}$ of the form 

\[
HG[\mathcal{I}(P_0)](t) = \Omega \left[1 + \sum_{q=1}^{\infty} \sum_{r=0}^q (-1)^q y^r_{0,q}(t) p^{q-r} \hbar^{-q} \right],
\]

\noindent where $y^r_{0,q}(t)$ are degree-$r$ symmetric polynomials in $\{\lambda_i\}$.  Then $P_0$ may be extended to $(0,1,p,\Omega)$-Euler data $P$ in such a way that the hypergeometric data at each height $k$ has series expansion of the form

\[
HG[\mathcal{I}(P_k)](t) = \Omega p^k \left[1 + \sum_{q=1}^{\infty} \sum_{r=0}^{q+k} (-1)^q y^r_{k,q}(t) p^{q-r} \hbar^{-q} \right],
\]

\noindent where $y^r_{k,q}(t)$ are degree-$r$ symmetric polynomials in $\{\lambda_i\}$ and the $y^0_{i,q}(t)$ are recursively defined by

\[
y^0_{i,q}(t) = \frac{{y^0_{i-1,q+1}}'(t)}{{y^0_{i-1,1}}'(t)}
\]

\noindent for $1 \leq i \leq q$.
\label{maintheorem}
\end{theorem}

\begin{proof}
We construct the $(0,1)$-Euler data $P$ by induction on height.  The existence of height $0$ data satisfying the theorem is presupposed.  Suppose that for all $0 \leq k' \leq k$, $P$ has been constructed with hypergeometric data 

\[
HG[\mathcal{I}(P_k)](t) = \Omega p^k \left[1 + \sum_{q=1}^{\infty} \sum_{r=0}^{q+k} (-1)^q y^r_{k,q}(t) p^{q-r} \hbar^{-q} \right],
\]

\noindent and we wish to construct the height $k+1$ data.  The operator $\rho: \mathcal{S} \rightarrow \mathcal{S}$ from Lemma \ref{extensionlemma} on $P_k$ gives height $k+1$ data $\rho(P_k)$ such that 

\begin{align*}
HG[\mathcal{I}(\rho(P_k))](t) &= -\hbar \frac{\partial}{\partial t} \left(HG[\mathcal{I}(P_k)](t) \right)\\
 &= \Omega p^k \left[\sum_{r=0}^{k+1}{y^r_{k,1}}'(t) p^{1-r} + \sum_{q=2}^{\infty} \sum_{r=0}^{q+k}(-1)^{q+1} {y^r_{k,q}}'(t) p^{q-r} \hbar^{-q+1} \right].
\end{align*}

Since $\Omega$ is assumed to be invertible, we see from the definition for hypergeometric data that $y_{k,q}^0(t) - \frac{t^q}{q!} \in e^t\mathcal{R}_T[[e^t]]$ and in particular $y_{k,1}^0(t) - t \in e^t\mathcal{R}_T[[e^t]]$, which implies ${y_{k,1}^0}'(t) - 1 \in e^t\mathcal{R}_T[[e^t]]$.  Thus $f = \hbar \ln({y_{k,1}^0}'(t)) \in e^t\mathcal{R}_T[[e^t]]$ and we may use Lemma \ref{trans2} to obtain a mirror transformation $\nu$ such that 

\begin{align*}
HG[\mathcal{I}(\nu(\rho(P_k)))](t) &= e^{-f/\hbar}HG[\mathcal{I}(\rho(P_k))](t)\\
 &= \Omega \left[p^{k+1} + \sum_{r=1}^{k+1} \frac{{y^r_{k,1}}'(t)}{{y^0_{k,1}}'(t)} p^{k+1-r}\right] + \Omega p^k \sum_{q=2}^{\infty} \sum_{r=0}^{q+k}(-1)^{q+1} \frac{{y^r_{k,q}}'(t)}{{y^0_{k,1}}'(t)} p^{q-r} \hbar^{-q+1} .
\end{align*}

The previous observation also implies that $\frac{-{y^r_{k,1}}'(t)}{{y^0_{k,1}}'(t)} \in e^t\mathcal{R}_T[[e^t]]$ for all $r=1...k+1$.  We may then use Lemma \ref{trans3} (repeatedly) to obtain another mirror transformation $\eta$ such that 

\begin{align*}
HG[\mathcal{I}(\eta(\nu(\rho(P_k))))](t) &= HG[\mathcal{I}(\nu(\rho(P_k)))](t) + \sum_{r=1}^{k+1} \frac{-{y^r_{k,1}}'(t)}{{y^0_{k,1}}'(t)} HG[\mathcal{I}(P_{k+1-r})](t)\\
  &= \Omega p^{k+1} + \Omega p^k \sum_{q=2}^{\infty} \sum_{r=0}^{q+k}(-1)^{q+1} \frac{{y^r_{k,q}}'(t)}{{^0_{k,1}}'(t)} p^{q-r} \hbar^{-q+1} \\
  & \qquad + \sum_{r=1}^{k+1} \Omega p^{k+1-r} \frac{-{y^r_{k,1}}'(t)}{{y^0_{k,1}}'(t)} \sum_{q=1}^{\infty} \sum_{s=0}^{q+k+1-r} (-1)^q y^s_{k+1-r,q}(t) p^{q-s} \hbar^{-q} \\
  &= \Omega p^{k+1} \left[1 + \sum_{q=1}^{\infty} \sum_{r=0}^{q+k+1} (-1)^q y^r_{k+1,q}(t) p^{q-r} \hbar^{-q} \right]\\
\end{align*}

\noindent for some functions $y^r_{k+1,q}(t)$.  It is clear that $y^r_{k+1,q}(t)$ is symmetric of degree $r$ in $\{\lambda\}$, and that $y^0_{k+1,q}(t)$ is given by

\[
y^0_{k+1,q}(t) = \frac{{y^0_{i-1,q+1}}'(t)}{{y^0_{i-1,1}}'(t)}.
\]

\noindent Thus we define $P_{k+1}$ to be $\eta(\nu(\rho(P_k)))$ and the theorem is proved.
\qed
\end{proof}

\section{Mirror Theorems}

In this section we combine the results of previous sections to establish mirror theorems, which express the Euler data coming from a concavex bundle in terms of linear classes on $N_d$ in the case of one and two marked points.  We first show that the $(0,1)$-Euler data induced by a concavex bundle $V$ and the extension constructed in Theorem \ref{maintheorem} are the same.  The case of two markings is solved by gluing together the results of the one point case.

\subsection{One-point Mirror Theorem}

The mirror theorems in the case of one marked point will follow from the following lemma, which explicitly identifies $I_d^*(\hat Q_{\vec{k}})$.

%
%
%
%

\begin{lemma}
Let $\phi^d = \prod_{l=0}^n \prod_{m=1}^d (p - \lambda_l - m\hbar) \in H^*_T(\pro^n)$ denote the denominator of the coefficient of $e^{dt}$ in the hypergeometric series $HG[\mathcal{I}(\hat Q_{\vec{k}})](t)$.  Then:

\begin{enumerate}
\item If $\pi_j$, $ev_j$ are as before then
\[
I^*_d(\hat Q_{d,\vec{k}}) = (-1)^v \phi^d p^{|\vec{k}''|} \hbar^{v-1}ev_{v+1*} \left( \frac{\pi_{v+1}^*(\text{Euler}_T(V_d) \prod_{j=1}^{v} ev_j^*p^{k_j})}{\hbar + c_1^T(\mathcal{L}_{d,v+1})}   \right)
\]

\item $\hat Q$ satisfies $\text{deg}_{\hbar} I_d^*(\hat Q_{d,\vec{k}}) \leq (n+1)d + v - 2$ for every $\vec{k}$.
\end{enumerate}
\label{recovery}
\end{lemma}

\begin{proof}
To prove the first part of the lemma, we show that both sides of the equation have equal localizations at $q_i \in \pro^n$ for all $0 \leq i \leq n$.  Since $I_d(q_i) = p_{i,0} \in N_d$, we have

\[
i^*_{q_i}(I^*_d(\hat Q_{d,\vec{k}})) = i^*_{p_{i,0}}(\hat Q_{d,\vec{k}}) = \hat Q_{d,\vec{k}}(\lambda_i).
\]

\noindent From the proof of Lemma \ref{showeulerdata}, we know 

\[
\hat Q_{d,\vec{k}}(\lambda_i) = (-1)^v i^*_{p_{i,0}}(\phi_{i,0}) \hbar^{v-1} \prod_{j=1}^u \lambda_i^{k_{j+v}} \sum_{\{F_d^{v+1}\}} \int_{(F_d^{v+1})_T} \frac{\pi_{v+1}^*(\text{Euler}_T(V_{d}) \prod_{j=1}^{v}ev_j^*p^{k_j})}{\text{Euler}_T(N(F_d^{v+1}))(\hbar + c_1^T(\mathcal{L}_{d,v+1}))}\\
\]

\noindent It follows from the definition of $\phi^d$ and $\phi_{q_i} = \prod_{j \neq i}(p-\lambda_j)$ that

\[
i^*_{p_{i,0}}(\phi_{p_{i,0}}) = i^*_{q_i}(\phi_{q_i})i^*_{q_i}(\phi^d),
\]

\noindent which, when combined with the above expression for $\hat Q_{d,k}(\lambda_i)$, gives

\[
\hat Q_{d,\mathbf{k}}(\lambda_i) = (-1)^v i^*_{q_i}(\phi_{q_i})i^*_{q_i}(\phi^d) \hbar^{v-1} \prod_{j=1}^u \lambda_i^{k_{j+v}} \sum_{\{F_d^{v+1}\}} \int_{(F_d^{v+1})_T} \frac{\pi_{v+1}^*(\text{Euler}_T(V_{d}) \prod_{j=1}^{v}ev_j^*p^{k_j})}{\text{Euler}_T(N(F_d^{v+1}))(\hbar + c_1^T(\mathcal{L}_{d,v+1}))}\\,
\]

\noindent where the summation is over fixed components $F_d^{v+1} \subset \overline{M_{0,v+1}}(\pro^n,d)$ such that the final marked point is mapped to $q_i \in \pro^n$.

On the other hand, if we let

\[
\alpha_d = \frac{\pi_{v+1}^*(\text{Euler}_T(V_d) \prod_{j=1}^{v}p^{k_j})}{\hbar + c_1^T(\mathcal{L}_{d,v+1})},
\]
 
\noindent then

\begin{align*}
i^*_{q_i}((-1)^v \hbar^{v-1} p^{|\vec{k}''|} \phi^d ev_{v+1*}(\alpha_d)) &= (-1)^v \hbar^{v-1} \int_{(\pro^n)_T} \phi_{q_i} p^{|\vec{k}'|} \phi^d ev_{v+1!}(\alpha_d)\\
 &= (-1)^v \hbar^{v-1} \int_{(\overline{M_{0,v+1}}(\pro^n,d))_T} ev_{v+1}^*(p^{|\vec{k}'|}\phi_{q_i} \phi^d) \alpha_d\\
 &=(-1)^v \hbar^{v-1} i^*_{q_i}(\phi_{q_i})i^*_{q_i}(\phi^d) \prod_{j=1}^u \lambda_i^{k_{j+v}} \int_{(\overline{M_{0,v+1}}(\pro^n,d))_T} \alpha_d.
\end{align*}

\noindent The last line follows since $ev_{v+1}^*\phi_{q_i}$ vanishes unless the image of that last point is $q_i$.  Hence, in calculating this expression via localization we only need consider fixed components of $\overline{M_{0,v+1}}(\pro^n,d)$ such that the last marked point is mapped to $q_i$.  Such components are precisely the ones considered in the proof of Lemma \ref{showeulerdata} to obtain the expression for $\hat Q_{d,\vec{k}}(\lambda_i)$.  Thus both sides have equal localizations at each fixed point; the lemma follows.

To show the degree bound in the second statement of the lemma, we know that $\hat Q_{d,\vec{k}}$, and hence $I_d^*(\hat Q_{d,\vec{k}})$, are both automatically polynomials in $\hbar$ since $\hat Q$ is $(u,v)$-Euler data.  It is clear that $\text{deg}_{\hbar}\phi^d = (n+1)d$, so part two of the lemma then follows immediately from the first part. \qed
\end{proof}

\noindent As an immediate application, we have the following theorem:

\begin{theorem}[Mirror Theorem I]

The $(0,1,p,\Omega)$-Euler data $\hat Q = \{Q_{d,k}\}$ induced by concavex bundle $V \rightarrow \pro^n$ from Lemma \ref{showeulerdata} and the extension $\hat P$ of its height $0$ data $\hat P_0 (=\hat Q_0)$ given by Theorem \ref{maintheorem} are equal.
\label{oneptmirrorthrm}
\end{theorem}

\begin{proof}
We confirm that the two conditions for uniqueness in Lemma \ref{uniquenesslemma} are met: $\hat Q$ and $\hat P$ have the same height $0$ data by assumption, which satisfies the hypotheses of Theorem \ref{maintheorem} by the second part of Lemma \ref{recovery}.  For $k>0$ the extension of Theorem \ref{maintheorem} satisfies

\[
HG[\mathcal{I}(\hat P_k)](t) \equiv \Omega p^k \mod \hbar^{-1}.
\]

\noindent The second statement in Lemma \ref{recovery} also implies that in the case of $(0,1)$-Euler data,

\[
HG[\mathcal{I}(\hat Q_k)](t) \equiv \Omega p^k \mod \hbar^{-1}.
\]

\noindent The degree bound in Lemma \ref{uniquenesslemma} is equivalent to requiring that the hypergeometric series agree modulo $\hbar^{-1}$, so the theorem follows. \qed
\end{proof}

\begin{remark}The entire development of $(0,1)$-Euler data can also be carried out for $(1,0)$-Euler data with the modification that the map $I_d: \pro^n \hookrightarrow N_d$ be replaced with $I_d': \pro^n \hookrightarrow N_d$, given by 

\[
I_d'([a_0,...,a_n]) = [a_0w_0^d,...,a_nw_0^d].
\]

\noindent Any statement concerning $(0,1)$-Euler data may then be read as a statement about $(1,0)$-Euler data by using $I_d'$.  In particular, an analogous one-point mirror theorem concerning the $(1,0)$-Euler data from a concavex bundle holds.
\end{remark}

\subsection{Two-point Mirror Theorem}

Explicitly identifying the $(1,1)$-Euler data induced by $V$ in terms of linear classes is acomplished by gluing together $(1,0)$ and $(0,1)$-Euler data.

\begin{theorem}[Mirror Theorem II]
The $(1,1)$-Euler data $\hat Q$ induced by a concavex bundle $V$ is expressible in terms of linear classes by gluing together its $(0,1)$-Euler data, identified by Theorem \ref{oneptmirrorthrm} in terms of linear classes, and its analog for $(1,0)$-Euler data as prescribed by the Euler condition (\ref{eulercondition}).  As a consequence, the associated hypergeometric data at height $\vec{k} = (k_1,k_2)$ is given by

\[
HG[\mathcal{I}(\hat Q_{d,\vec{k}})](t) = p^{k_2}HG[\mathcal{I}(\hat Q_{d,k_1})](t)
\]
\label{twoptmirrorthrm}
\end{theorem}

\begin{proof}
As an equivariant class is fully determined by its values at fixed components, the Euler condition (\ref{eulercondition}) identifies any $(1,1)$-Euler data in terms of its $(1,0)$-Euler data and $(0,1)$-Euler data.  A good solution to the one-point problem as given by Theorem \ref{oneptmirrorthrm} then identifies $\hat Q$ in terms of linear classes on $N_d$.  The second statement follows from the fact that $I_d^*(Q_k) = p^kI_d^*(Q_0)$ for any $(1,0)$-Euler data $Q$. \qed
\end{proof}

\section{Recovery Lemmas}

Recall that the Gromov-Witten invariants $K_{d,\vec{k}}$ associated with the concavex bundle $V$ on $\pro^n$ are defined as

\begin{equation}
K_{d,\vec{k}} = \int_{\overline{M_{0,m}}(\pro^n,d)} \text{Euler}(V_d) \prod_{j=1}^m ev^*H^{k_j}
\label{gwdef2}
\end{equation}

\noindent for $\vec{k} = (k_1,..., k_m)$ whenever

\[
\text{dim } \overline{M_{0,m}}(\pro^n,d) = \text{dim } V_d + |\vec{k}|,
\]

\noindent in which case we say that the integrand is of \emph{critical dimension}.  

The next lemma is useful for extracting the invariants $K_{d,\vec{k}}$ from hypergeometric data.

\begin{lemma}  If $\hat Q$ is $(0,v)$-Euler data induced by a concavex bundle $V$ and $\vec{k} = (k_1,...,k_v)$ is such that the integrand in (\ref{gwdef2}) is of critical dimension, then

\begin{equation}
\lim_{\lambda_i \rightarrow 0} \int_{(\pro^n)_T} p^s e^{-pt/{\hbar}} \frac{I_d^*(\hat Q_{d,\vec{k}})}{\prod_{l=0}^n \prod_{m=1}^d (p - \lambda_l - m\hbar)} = \left\{
     \begin{array}{lr}
       (-1)^v\hbar^{v-3}(2-v-dt)K_{d,\vec{k}} & s=0\\
       (-1)^v\hbar^{v-2}d K_{d,\vec{k}} & s=1\\
       0 & s > 1
     \end{array}
     \right\}
\label{cases}
\end{equation}

\label{caseslemma}

\end{lemma}

\begin{proof}
Denote the left-hand side of (\ref{cases}) by $A_s$.  By Lemma \ref{recovery}, 

\[
\lim_{\lambda_i \rightarrow 0} I_d^* \hat Q_{d,\vec{k}}  = (-1)^v \hbar^{v-1} \prod_{m=1}^d (H-m\hbar)^{n+1} ev_{v+1*} \left( \frac{\pi_{v+1}^*(\text{Euler}(V_d) \prod_{j=1}^v ev_j^*H^{k_j})}{\hbar + c_1(\mathcal{L}_{d,v+1})} \right).
\]

\noindent Inserting this expression into $A_s$, we have

\begin{align*}
A_{s} &= (-1)^v \hbar^{v-1} \int_{\pro^n} H^s e^{-Ht/\hbar} ev_{v+1*} \left( \frac{\pi_{v+1}^*(\text{Euler}(V_d) \prod_{j=1}^v ev_j^*H^{k_j})}{\hbar + c_1(\mathcal{L}_{d,v+1})}   \right)\\
		&= (-1)^v \hbar^{v-1}\int_{\overline{M_{0,v+1}}(\pro^n,d)} ev_{v+1}^*H^s e^{-ev_{v+1}^*Ht/\hbar} \cdot  \frac{\pi_{v+1}^*(\text{Euler}(V_d) \prod_{j=1}^v ev_j^*H^{k_j})}{\hbar + c_1(\mathcal{L}_{d,v+1})}\\
		&= (-1)^v \hbar^{v-1} \int_{\overline{M_{0,v}}(\pro^n,d)} \text{Euler}(V_d) \prod_{j=1}^v ev_j^*H^{k_j} \pi_{v+1*} \left[\frac{ev_{v+1}^*H^s e^{-ev_{v+1}^*Ht/\hbar}}{\hbar + c_1(\mathcal{L}_{d,v+1})} \right].\\
\end{align*}

Since $\text{Euler}(V_d) \prod_{j=1}^v ev_j^*H^{k_j}$ has top degree, this integral is just $K_{d,\vec{k}}$ times a scalar factor given by integrating the expression in brackets over a generic fiber of $\pi_{v+1}$ (which is isomorphic to $\pro^1$).  Hence, 

\begin{equation}
A_{s} = (-1)^v \hbar^{v-1} K_{d,\vec{k}} \int_{\pro^1} \frac{ev_{v+1}^*H^s}{\hbar} \left( \frac{-c_1(\mathcal{L}_{d,2})}{\hbar} - ev_{v+1}^*H\frac{t}{\hbar} \right)
\label{cases2}
\end{equation}

The class $c_1(\mathcal{L}_{d,v+1})$ restricted to a fiber is just the cotangent bundle of $\pro^1$ with $v$ marked points, while the evaluation map on the fiber is just $f$ which has degree $d$.  By the Gauss-Bonnet theorem, the former class gives a contribution of  

\[
\int_{\pro^1} i^*c_1(\mathcal{L}_{d,v+1}) = -2 + v,
\]

\noindent while the latter contributes

\[
\int_{\pro^1} i^*ev_{v+1}^*H = d.
\]

Inserting these into (\ref{cases2}) and considering the various possibilities for $s$ establishes the lemma. \qed
\end{proof}

We will need a general recovery lemma for $(1,v)$-Euler data as well.  Fix $\vec{k} = (k_1,...,k_v,k_{v+1})$ such that the integrand in

\begin{equation*}
K_{d,\vec{k}} = \int_{\overline{M_{0,v+1}}(\pro^n,d)} \text{Euler}(V_d) \prod_{j=1}^{v+1} ev_j^*H^{k_j}
\end{equation*}

\noindent is of critical dimension.  Let $0 \leq k_{v+1}^* \leq k_{v+1}$ and $\vec{k}^* = (k_1,...,k_{v+1}^*)$.  Define

\begin{equation}
K_{d,\vec{k}}^i = \int_{\overline{M_{0,v+1}}(\pro^n,d)} \text{Euler}(V_d) \prod_{j=1}^{v} ev_j^*H^{k_j} \cdot ev_{v+1}^*H^i \cdot \psi_{v+1}^{k_{v+1}-i}
\label{gwdefpsi}
\end{equation}

\noindent where as before $\psi_i$ is the first Chern class of the cotangent line bundle.  In particular, $K_{d,\vec{k}}^0 = K_{d,\vec{k}}$.

\begin{lemma}
If $\hat Q$ is $(1,v)$-Euler data induced by a concavex bundle $V$ on $\pro^n$ and $k_{v+1}$, $k_{v+1}^*$, $\vec{k}$, and $\vec{k}^*$ are defined as above and $K_{d,\vec{k}}^i$ as in (\ref{gwdefpsi}), then 

\[
\lim_{\lambda_i \rightarrow 0} \int_{(\pro^n)_T} p^s e^{-pt/{\hbar}} \frac{I_d^*(\hat Q_{d,\vec{k}^*})}{\prod_{l=0}^n \prod_{m=1}^d (p - \lambda_l - m\hbar)} = (-1)^{v+a} \hbar^{v-2-a} \sum_{j=0}^{a} K_{d,\vec{k}}^{a+j}  \frac{t^j}{j!},
\]

\noindent where $a=k_{v+1}-k_{v+1}^* -s$ when $0 \leq s \leq k_{v+1} - k_{v+1}^*$ and is $0$ otherwise.
\label{caseslemma2}
\end{lemma}

\begin{proof}
We again calculate the integral by pulling back to $\overline{M_{0,v+1}}(\pro^n,d)$ and using the expression for $I_d^*(\hat Q_{d,\vec{k}^*})$ given in Lemma (\ref{recovery}):

\begin{align*}
\lim_{\lambda_i \rightarrow 0} &\int_{(\pro^n)_T} p^s e^{-pt/{\hbar}} \frac{I_d^*(\hat Q_{d,\vec{k}^*})}{\prod_{l=0}^n \prod_{m=1}^d (p - \lambda_l - m\hbar)} \\
& =(-1)^v \hbar^{v-1} \int_{\overline{M_{0,v+1}}(\pro^n,d)} ev_{v+1}^*H^{s+k_{v+1}^*} e^{-ev_{v+1}^*Ht/{\hbar}} \cdot \frac{\pi_{v+1}^*(\text{Euler}(V_d) \prod_{j=1}^v ev_j^*H^{k_j})}{\hbar + c_1(\mathcal{L}_{d,v+1})}\\
& = (-1)^v \hbar^{v-2-k_{v+1}+k_{v+1}^* + s} \sum_{j=0}^{k_{v+1} - (k_{v+1}^*+s)}(-1)^{k_{v+1} - (k_{v+1}^*+s)} K_{d,\vec{k}}^{k_{v+1}-(k_{v+1}^*+j+s)}  \frac{t^j}{j!},
\end{align*}

\noindent where we have expanded everything by power series and taken the terms of critical dimension.  The lemma then follows.  \qed
\end{proof}

\section{Applications}

\subsection{One-Point Candelas Formula}

We will prove a one-point generalization for hypersurfaces of the famous Candelas formula, which expresses mirror symmetry for the quintic in the unmarked case.

Let $V = \mathcal{O}(n+1) \rightarrow \pro^n$ with $n \geq 4$.  The $(n+1)p$-Euler data 

\begin{equation}
\hat P_0 = \prod_{m=0}^{(n+1)d}((n+1)\kappa - m\hbar)
\label{pzero}
\end{equation}

\noindent has hypergeometric data given in the non-equivariant limit in power series expansion as

\[
\lim_{\lambda \rightarrow 0} HG[\mathcal{I}(\hat P_0)](t) = (n+1)p \sum_{q=0}^{\infty} (-1)^q f_{q}(t) p^{q} \hbar^{-q}.
\]

\noindent Define $y_{i,q}(t)$ recursively by 

\begin{equation}
y_{0,q}(t) = \frac{f_{q}(t)}{f_{0}(t)}; \qquad y^0_{i,q}(t) = \frac{{y^0_{i-1,q+1}}'(t)}{{y^0_{i-1,1}}'(t)} \qquad i \geq 1
\label{recursionsfory}
\end{equation}

\noindent For dimensional reasons, the one-point Gromov-Witten invariants of interest are

\[
K_d(H^{n-3}) = \int_{\overline{M_{0,1}}(\pro^n,d)} \text{Euler}(V_d) ev_1^*H^{n-3}.
\]

\begin{lemma}[One-point Candelas Formula for Hypersurfaces]

Let the Gromov-Witten potential $\Phi$ for the degree $n+1$ hypersurface in $\pro^n$ be defined 

\[
\Phi(t) = \frac{(n+1)}{2}t^2 + \sum K_d(H^{n-3})e^{dt}.
\]

Then 

\[
\Phi(T) = (n+1)\left[y_{0,1}(t)y_{n-3,1}(t) - y_{n-3,2}(t) \right]
\]

\noindent where the $y_{i,q}(t)$ are defined by (\ref{recursionsfory}) and the mirror change of variables is

\[
T(t) = y_{0,1}(t).
\]

\end{lemma}

\begin{proof}
Let $\hat Q_0$ be the $(0,0)$-Euler data induced by $V=\mathcal{O}(n+1)$ and $\hat P_0$ be as in (\ref{pzero}).  As a consequence of Lemma $3.3$ of \cite{lly:mirror1}, mirror transformations of $\hat Q_0$ and $\hat P_0$ agree, which we will denote $\mu(\hat Q_0)$ and $\nu(\hat P_0)$, in such a way that

\[
HG[\mathcal{I}(\mu(\hat Q_0))(t) = HG[\mathcal{I}(\hat Q_0)(T(t)) = \frac{1}{f_0} HG[\mathcal{I}(\hat P_0)(t) = HG[\mathcal{I}(\nu(\hat P_0))(t)
\]

\noindent in the nonequivariant limit.  By Theorem \ref{oneptmirrorthrm}, the extension of $\nu(\hat P_0)$ to $(0,1)$-Euler data $\nu(\hat P)$ given by Theorem \ref{maintheorem} and the transformed $(0,1)$-Euler data $\mu(\hat Q)$ agree (here $\mu$ now refers to the extension given in Lemma \ref{trans1} of the mirror transformation of \cite{lly:mirror1} to all heights).  In particular, again working in the nonequivariant limit,

\begin{align*}
HG[\mathcal{I}(\mu(\hat Q_{n-3}))(t) &= HG[\mathcal{I}(\hat Q_{n-3})(T)\\
&= \Omega p^{n-3} \left[1 - \frac{1}{n+1} \Phi'(T) \frac{p}{\hbar} + \frac{1}{n+1}(T\Phi'(T) - \Phi(T)) \frac{p^2}{\hbar^2}\right]
\end{align*}

\noindent agrees with

\[
HG[\mathcal{I}(\nu(\hat P_{n-3})) = \Omega p^{n-3}\left[1- y_{n-3,1}\frac{p}{\hbar} + y_{n-3,2}\frac{p^2}{\hbar^2}\right].
\]

\noindent The above expression for $HG[\mathcal{I}(\hat Q_{n-3})(T)$ is a direct consequence of Lemma \ref{caseslemma}.  The theorem then follows easily by equating terms in these expansions.\qed

\end{proof}

\subsection{One-Point Closed Formula for $rk(V^-) \geq 2$}

Let $V = V^+ \oplus V^-$ be a split concavex bundle on $\pro^n$ such $\sum l_i + \sum k_i = n+1$, and the rank of the concave part is at least two:

\[
V = \bigoplus_{l_i > 0} \mathcal{O}(l_i) \bigoplus_{k_i > 0} \mathcal{O}(-k_i),
\]

\noindent with at least two $k_i$ factors.  When 

\begin{equation}
k = n - 2 - rk(V^+) + rk(V^-) > 0
\label{krank}
\end{equation}

\noindent the integrand in

\begin{equation}
K_d(H^k) = \int_{\overline{M_{0,1}}(\pro^n,d)} \text{Euler}(V_d) \text{ev}_1^*H^k
\label{nomirror}
\end{equation}

\noindent has the correct dimension.

\begin{lemma}
The one-point Gromov-Witten invariants $K_d(H^k)$ in \ref{nomirror} with $k$ as in \ref{krank} and rk$(V^-) = 2$ are given by

\[
K_d(H^k) =(-1)^{d(\sum k_i)} \frac{\prod l_i \prod_{l_i}(l_id)! \prod_{k_i}(k_id -1)! }{(d!)^{n+1}}
\]

\noindent and are $0$ if rk$(V^-) > 2$.

\label{onepointconvex}

\end{lemma}

As a special case, when $V = \mathcal{O}(-1) \oplus \mathcal{O}(-1) \rightarrow \pro^1$, we recover the multiple covering formula for one marked point mentioned in the introduction:

\[
K_{d,1} = \frac{1}{d^2}
\]

\noindent Although this contribution follows easily from the original Aspinwall-Morrison formula and the Divisor Equation \cite{mirrorsymmetry}, the methods here give another simple proof.  Other specific cases agree with calculations appearing elsewhere, for instance $V = \mathcal{O}(-1) \oplus \mathcal{O}(-2) \rightarrow \pro^2$ considered in Section 3.2 of \cite{klemmpand:4folds} (where the authors' ultimate interest was genus 1 invariants).

We now prove Lemma \ref{onepointconvex}:

\begin{proof}
%

As before, let $\Omega = \frac{e(V^+)}{e(V^-)}$.  If $\hat Q_0$ is the $(0,0,p,\Omega)$-Euler data induced by $V$ and $\hat P_0$ is the $(0,0,p,\Omega)$-Euler data given by

\[
\hat P_d = \prod_{l_i} \prod_{m=0}^{l_id}(l_i\kappa - m\hbar) \prod_{k_i} \prod_{m=1}^{k_id-1}(-k_i\kappa + m\hbar)
\]

\noindent then by the results of \cite{lly:mirror1}, $\hat Q_0$ and $\hat P_0$ are linked and, because $rk(V^-) \geq 2$, immediately satisfy

\[
HG[\mathcal{I}(Q_0)](t) \equiv HG[\mathcal{I}(P_0)](t) \mod \hbar^{-2},
\]

\noindent from which one concludes that $\hat Q_0$ = $\hat P_0$ (Lemma \ref{llyuniqueness}).

Now let $\hat Q$ be the $(1,0,p,\Omega)$-Euler data induced by $V$ (Lemma \ref{showeulerdata}).  Notice that by Lemma \ref{recovery},

\begin{align*}
I_d^*(\hat Q_k) &= p^k I_d^*(\hat Q_0).
\end{align*}

\noindent By inserting this into Lemma \ref{caseslemma2} in the case of $(1,0)$-Euler data, we find

\begin{align*}
\hbar^{-2} K_d(H^n) &= \lim_{\lambda_i \rightarrow 0} \int_{(\pro^n)_T} e^{-pt/{\hbar}} \frac{I_d^*(\hat Q_{d,k})}{\prod_{l=0}^n \prod_{m=1}^d (p - \lambda_l - m\hbar)}\\
 &= \lim_{\lambda_i \rightarrow 0} \int_{(\pro^n)_T} e^{-t/{\hbar}} \frac{p^k I_d^*(\hat Q_{d,0})}{\prod_{l=0}^n \prod_{m=1}^d (p - \lambda_l - m\hbar)}\\
&=\lim_{\lambda \rightarrow 0} \int_{(\pro^n)_T} e^{-pt/\hbar} \frac{p^k \prod_{l_i} \prod_{m=0}^{l_id}(l_ip - m\hbar) \prod_{k_i} \prod_{m=1}^{k_id-1}(-k_ip + m\hbar)}{\prod_{l=0}^n \prod_{m=1}^d (p - \lambda_l - m\hbar)}\\
&= (-1)^{d(\sum k_i)} \frac{\prod l_i \prod_{l_i}(l_id)! \prod_{k_i}(k_id -1)! }{\hbar^2(d!)^{n+1}}
\end{align*}

\noindent where the last line follows by picking off the term of proper degree when rk$(V^-)=2$.  If rk$(V^-)>2$ then the contribution is $0$. \qed

\end{proof}

\subsection{Generalized Multiple Covering Formula}

Let $V = \mathcal{O}(-1)^{n+1}$ on $\pro^n$, which induces on $\overline{M_{0,2}}(\pro^n,d)$ a sequence of obstruction bundles $V_d$ of dimension $(n+1)d - n - 1$.  Since $\overline{M_{0,2}}(\pro^n,d)$ has dimension $(n+1)d + n - 1$, the only two-point invariants for dimensional reasons are 

\[
K_d(H^n,H^n) = \int_{\overline{M_{0,2}}(\pro^n,d)} \text{Euler}(V_d)ev_1^*H^n ev_2^*H^n.
\]

\begin{lemma}
\[
K_d(H^n,H^n) = (-1)^{(n+1)(d-1)} \frac{1}{d}
\]
\end{lemma}

\noindent This agrees with the calculation made in \cite{yplee:flops}.

\begin{proof}
The $(0,0,(-p)^{-n-1})$-Euler data $\hat P_0 = \prod_{m=1}^{d-1}(-p+m\hbar)^{n+1}$ is immediately equal to the $(0,0,(-p)^{-n-1})$-Euler data $\hat Q_0$ induced by $V$ as they are linked and their hypergeometric series agree mod $\hbar^{-2}$ \cite{lly:mirror1}.  The simple extension of $\hat P_0$ to $(0,1)$-Euler data of Lemma \ref{extensionlemma} given by $I_d^*\hat P_{d,k} = (p-d\hbar)^kI_d^*\hat P_{d,0}$ already has the property that 

\[
HG[\mathcal{I}(\hat P_n)](t) \equiv \frac{(-1)^{n+1}}{p} \mod \hbar^{-1}.
\]

On the other hand, the $(0,1)$-Euler data $\hat Q$ induced by $V$ satisfies

\[
HG[\mathcal{I}(\hat Q_n)](t) \equiv \frac{(-1)^{n+1}}{p} \mod \hbar^{-1}
\]

\noindent as well by Lemma \ref{recovery}.  These two $(0,1)$-Euler data then agree by Lemma \ref{uniquenesslemma}.  Theorem \ref{twoptmirrorthrm} then yields that the $(1,1)$-Euler data $\hat Q$ induced by $V$ satisfies $HG[\mathcal{I}(\hat Q_{n,n})](t) = p^n HG[\mathcal{I}(\hat P_n)](t) = p^n HG[\mathcal{I}(\hat P_n)](t)$.  One may then calculate

\begin{align*}
-\hbar^{-1} K_d(H^n,H^n) &= \lim_{\lambda_i \rightarrow 0} \int_{(\pro^n)_T} e^{-pt/{\hbar}} \frac{I_d^*(\hat Q_{d,n,n})}{\prod_{l=0}^n \prod_{m=1}^d (p - \lambda_l - m\hbar)}\\
&=\lim_{\lambda \rightarrow 0} \int_{(\pro^n)_T} e^{-pt/\hbar} p^n \frac{(p-d\hbar)^n \prod_{m=1}^{d-1}(-p+m\hbar)^{n+1}}{\prod_{l=0}^n \prod_{m=1}^d (p - \lambda_l - m\hbar)}\\
&= (-1)^{(n+1)(d-1)+1} \frac{1}{d\hbar}
\end{align*}

\noindent per Lemma \ref{caseslemma2}, from which the lemma follows. \qed

\end{proof}

\subsection{Example of General Case}

We now give an example typical of the general case computable by the methods in this paper.  Consider the bundle $V = \mathcal{O}(3) \oplus \mathcal{O}(-3) \rightarrow \pro^5$, which defines a sequence of bundles $V_d$ on $\overline{M_{0,1}}(\pro^5,d)$ (or $\overline{M_{0,2}}(\pro^5,d)$) of rank $6d$.  Since $\overline{M_{0,1}}(\pro^5,d)$ has dimension $6d + 3$ and $\overline{M_{0,1}}(\pro^5,d)$ has dimension $6d+4$, we should calculate 

\[
K_d(\tau_i(H^{3-i})) = \int_{\overline{M_{0,1}}(\pro^5,d)} \text{Euler}(V_d) ev_1^*H^{3-i} \psi_1^i.
\]

\noindent for $0 \leq i \leq 3$ and

\[
K_d(H^2,\tau_i(H^{2-i})) = \int_{\overline{M_{0,2}}(\pro^5,d)} \text{Euler}(V_d) ev_1^*H^{2} ev_2^*H^{2-i} \psi_2^i.
\]

\noindent for $0 \leq i \leq 2$.  The other invariants computable by the methods in this paper follow quickly from the divisor property or symmetry and are not included.

Let $\hat Q = \{\hat Q_{d,k}\}$ be the $(0,1,p,-1)$-Euler data induced by $V$ as in Lemma \ref{showeulerdata}.  Take $g(t) = \sum_{d \geq 1} \frac{(-1)^d(3d)!^2}{d(d!)^6}e^{dt} \in e^t\mathcal{R}_T[[e^t]]$; by Lemma \ref{trans1} there exists a mirror transformation $\mu$ so that for every height $k$,

\begin{equation}
HG[\mathcal{I}(\mu(\hat Q)_k)](t) = HG[\mathcal{I}(\hat Q_k)](t + g(t)).
\label{qside}
\end{equation}

\noindent In particular, by using Lemma \ref{recovery}, at height $k=0$ (corresponding to the $(0,0)$-Euler data induced by $V$) the associated hypergeometric data has the expansion 

\begin{align*}
HG[\mathcal{I}(\mu(\hat Q)_0)](t) &= HG[\mathcal{I}(\hat Q_0)](t + g(t))\\
  &= -1 \left[1 - \frac{p}{\hbar}(t + g(t))\right] + O(\hbar^{-2}).
\end{align*}
						
Now consider the $(0,0)$-Euler data $\hat P_0$ defined by

\[
\hat P_{d,0} = \prod_{m=0}^{3d}(3\kappa - m\hbar) \prod_{m=1}^{3d-1}(-3\kappa + m\hbar),
\]

\noindent which is linked to $\hat Q_0$ (and hence $\mu(Q)_0$) by \cite{lly:mirror1}.  The hypergeometric data for $\hat P_0$ has series expansion of the form

\begin{align*}
HG[\mathcal{I}(\hat P_0)](t) &= \Omega \cdot \sum_{q=0}^{\infty} \sum_{r=0}^q (-1)^q y^r_{0,q}(t) p^{q-r} \hbar^{-q}\\
&= -1 \left[ 1 - \hbar^{-1}(p \cdot y_{0,1}^0(t) + y_{0,1}^1(t)) \right] + O(\hbar^{-2}) 
\end{align*}

Let $f(t) = y_{0,1}^1(t) \in e^t\mathcal{R}_T[[e^t]]$.  By Lemma \ref{trans2}, there exists a mirror transformation $\nu$ so that 

\begin{align*}
HG[\mathcal{I}(\nu(\hat P_0))](t) &= e^{f/\hbar} HG[\mathcal{I}(\hat P_0)](t)\\
							&= -1 \left[ 1 - \frac{p}{\hbar} \cdot y_{0,1}^0(t)\right] + O(\hbar^{-2}) .
\end{align*}

\noindent One may explicitly check that $y_{0,1}^0 = g(t) + t$, so that

\[
HG[\mathcal{I}(\mu(\hat Q)_0)](t + g(t)) \equiv HG[\mathcal{I}(\nu(\hat P_0))](t) \mod \hbar^{-2},
\]

\noindent which, by Lemma \ref{llyuniqueness}, implies that $\mu(\hat Q)_0 = \nu(\hat P_0)$ as $(0,0)$-Euler data.

By Theorem \ref{oneptmirrorthrm}, the extension of $\hat P_0$ of Theorem \ref{maintheorem} and the transformed $(0,1)$-Euler data $\mu(\hat Q)$ induced by $V$ agree.  For instance, this implies

\[
HG[\mathcal{I}(\hat Q_3)](t + g(t)) =\frac{-\hbar}{y_{3,3}^0} \frac{\partial}{\partial t} \left[ \frac{-\hbar}{y_{2,2}^0} \frac{\partial}{\partial t} \left[ \frac{-\hbar}{y_{1,1}^0} \frac{\partial}{\partial t} HG[\mathcal{I}(\hat P_0)](t)\right] \right].
\]

\noindent from which one can calculate the one-point Gromov-Witten invariants $K_d(H^3)$.

More generally, Theorem \ref{twoptmirrorthrm} along with the recovery lemmas \ref{caseslemma} and \ref{caseslemma2} may be used to recover one and two-point Gromov-Witten with descendents, which are calculated by computer and given in the accompanying figures.

\begin{center}
\begin{figure}
\centering
\begin{tabular}[t]{|l|c|c|}
\hline
$d$ & $K_d(H^3)$ & $\eta_d(H^3)$\\\hline
1 & $144$ & 144\\
2 & $-15228$ &  -15264\\
3 & $3387832$ & 3387816\\
4 & $-1033328799$ & -1033324992\\
5 & $\frac{9395106912144}{25}$ & 375804276480\\
6 & $152957189840958$ & -152957190686220\\
7 & $\frac{3299075934458784120}{49}$ & 67328080295077224\\
8 & $\frac{-125586661840964581023}{4}$ & -31396665459982813056\\
9 & $\frac{137738185029530693381824}{9}$ & 15304242781058965554888\\
10 & $\frac{-193162880799109330140903228}{25}$ & -7726515231964467156704640\\ \hline
\end{tabular}
\caption{One-point Gromov-Witten invariants}
\end{figure}
\end{center}

\begin{center}
\begin{figure}
\centering
\begin{tabular}[t]{|l|c|c|c|}
\hline
$d$ & $K_d(\tau_1(H^{2}))$ & $K_d(\tau_2(H))$ & $K_d(\tau_3(1))$\\\hline
1 & $27$ & 207 &  -414\\
2 & $\frac{136485}{8}$ &  $\frac{-339471}{16}$ &  $\frac{38799}{16}$\\
3 & $-4712813$ & $\frac{51696073}{18}$ &  $\frac{137491061}{108}$\\
4 & $\frac{100423232037}{64}$ & $\frac{-150374595087}{256}$ &  $\frac{-268540355185}{512}$\\
5 & $\frac{-74841919774848}{125}$ & $\frac{1483860161171187}{10000}$ &  $\frac{41652053617157379}{200000}$\\
6 & $\frac{50249578265872917}{200}$ & $\frac{-168846749581461909}{4000}$ &  $\frac{-21002329435853044853}{240000}$\\
7 & $\frac{-968660782674268810158}{8575}$ & $\frac{30592323981030547843299}{2401000}$ &  $\frac{13092926237088181035302337}{336140000}$\\
8 & $\frac{33559516570446549225317133}{627200}$ & $\frac{-678590997243490327156420149}{175616000}$ &  $\frac{-447547582585761567954290766657}{24586240000}$\\
9 & $\frac{-873077405632389938867605433}{33075}$ & $\frac{4338715330778600822100647627393}{4000752000}$ &  $\frac{89279689020110822840228699723304167}{10081895040000}$\\
10 & $\frac{2354511971663254407450821413}{175}$ & $\frac{-5193643584353587057315746796417}{23520000}$ &  $\frac{-791809569634596848121327607679163149}{177811200000}$\\ \hline
\end{tabular}
\caption{One-point Gromov-Witten invariants with descendents}
\end{figure}
\end{center}

\begin{center}
\begin{figure}
\centering
\begin{tabular}[t]{|l|c|c|}
\hline
$d$ & $K_d(H^2, H^2)$ & $\eta_d(H^2,H^2)$ \\\hline
1 & $261$ & 117 \\
2 & $\frac{-141669}{2}$ &  $-70965$\\
3 & $28141053$ & $28140966$ \\
4 & $\frac{-52415855109}{4}$ & $-13103928360$ \\
5 & $\frac{33295202406036}{5}$ & $6659040481155$ \\
6 & $\frac{-7152693165982701}{2}$ & $-3576346597038222$ \\
7 & $\frac{13967902006221361284}{7}$ & $1995414572317337289$ \\
8 & $\frac{-9158575282812536703813}{8}$ & $-1144821910345015106088$\\
9 & $670934720575777355387210$ & $670934720575777346006859$ \\
10 & $\frac{-1999257192955989705434315922}{5}$ & $-399851438591201270607089595$\\ \hline
\end{tabular}
\caption{Two-point Gromov-Witten invariants}
\end{figure}
\end{center}

\begin{center}
\begin{figure}
\centering
\begin{tabular}[t]{|l|c|c|}
\hline
$d$ & $K_d(H^2,\tau_1(H))$ & $K_d(H^2,\tau_2(1))$\\\hline
1  & -117 &  -27\\
2  &  $\frac{-31995}{4}$ &  $\frac{239355}{8}$\\
3 & $6726101$ &  $-10610038$\\
4 & $\frac{-62339468859}{16}$ &  $\frac{292304178171}{64}$\\
5 & $\frac{55140229054008}{25}$ &  $\frac{-1097087319721233}{500}$\\
6 & $\frac{-25282225051035849}{20}$ &  $\frac{113016673920662199}{100}$\\
7 & $\frac{180503575429845806631}{245}$ &  $\frac{-10471104834998308828011}{17150}$\\
8 & $\frac{-976993199472105955290897}{2240}$ &  $\frac{42829575696798982686013239}{125440}$\\
9 & $\frac{82446713274250766320045042}{315}$ &  $\frac{-14815808389405172763875125873}{75600}$\\
10 & $\frac{-5561428230594376375028839092}{35}$ &  $\frac{4499962304088589402305334576101}{39200}$\\ \hline
\end{tabular}
\caption{Two-point Gromov-Witten invariants with descendents}
\end{figure}
\end{center}

\newpage
\bibliographystyle{amsalpha}
\bibliography{refs}

\end{document}